\documentclass[12pt]{amsart}
\usepackage{amssymb}
\usepackage{amsmath}
\usepackage{bbm}
\usepackage{color}
\usepackage{verbatim}
\usepackage{pictex}

\definecolor{darkred}{rgb}{0.75,0,0.3}

\newcommand\AND{\quad\text{and}\quad}

\newcommand\C{\mathbb C}

\newcommand\cuplus{\,\hbox{$\cup^{^{\hspace*{-6.75pt}\scriptscriptstyle+}}$}\,}

\newcommand\diag{\mathsf{diag}}
\newcommand\dps{\displaystyle}

\newcommand\Lp{\mathfrak{L}}

\newcommand\Prob{\mathbb{P}}

\newcommand\R{\mathbb R}
\newcommand\res{\hspace*{1pt}\rule[-5pt]{.5pt}{10pt}\hspace*{1pt}}
\newcommand\ret{\hspace*{-2pt}\rule[-2.5pt]{0pt}{10pt}}

\newcommand\taub{\boldsymbol{\tau}}

\newcommand\uno{\mathbf{1}}
\newcommand\wh{\widehat}
\newcommand\wt{\widetilde}

\newcommand\Z{\mathbb Z}
\newcommand\zero{\mathbf{0}}

\numberwithin{equation}{section}

\newtheoremstyle{mythm}
  {9pt}
  {9pt}
  {\itshape}
  {0pt}
  {\bfseries}
  {}
  { }
  {\thmnumber{(#2)}\thmname{ #1}\thmnote{ #3}}

\newtheoremstyle{mydef}
  {9pt}
  {9pt}
  {\normalfont}
  {0pt}
  {\bfseries}
  {}
  { }
  {\thmnumber{(#2)}\thmname{ #1}\thmnote{ #3}}

\theoremstyle{mythm}
\newtheorem{thm}[equation]{Theorem.}
\newtheorem{pro}[equation]{Proposition.}
\newtheorem{lem}[equation]{Lemma.}
\newtheorem{cor}[equation]{Corollary.}

\theoremstyle{mydef}
\newtheorem{dfn}[equation]{Definition.}
\newtheorem{exa}[equation]{Example.}

\newtheorem{rmk}[equation]{Remark.}

\newtheorem{rmkdef}[equation]{Definition \& Remark.}
\newtheorem{dis}[equation]{Discussion.}
\newtheorem{imp}[equation]{}

\begin{document}$\,$ \vspace{-1truecm}
\title{Laplace and bi-Laplace equations for directed networks and Markov chains}
\author{\bf Thomas Hirschler, Wolfgang Woess}
\address{\parbox{.8\linewidth}{Institut f\"ur Diskrete Mathematik,\\ 
Technische Universit\"at Graz,\\
Steyrergasse 30, A-8010 Graz, Austria\\}}
\email{thirschler@tugraz.at, woess@tugraz.at}
\date{March 31, 2021} 
\begin{abstract}
The networks of this -- primarily (but not exclusively) expository -- compendium are strongly connected, finite directed
graphs $X$, where each oriented edge $(x,y)$ is equipped with a positive
weight (conductance) $a(x,y)$. We are not assuming symmetry of this
function, and in general we do not require that along with $(x,y)$,   
also $(y,x)$ is an edge. The weights give rise to a difference
operator, the normalised version of which we consider as our Laplace
operator. It is associated with a Markov chain with state space $X$. 
A non-empty subset of $X$ is designated as the boundary.
We provide a systematic exposition of the different types of Laplace
equations, starting with the Poisson equation, Dirichlet problem and Neumann problem.
For the latter, we discuss the definition of outer normal derivatives.
We then pass to Laplace equations involving potentials, thereby also addressing
the Robin boundary problem. 
Next, we study the bi-Laplacian and associated equations: the iterated Poisson
equation, the bi-Laplace Neumann and Dirichlet problems, and the ``plate equation''.
It turns out that the bi-Laplace Dirichlet to Neumann map is of non-trivial
interest.
The exposition concludes with two detailed examples.
\end{abstract}
                  \thanks{Partially supported by Austrian Science Fund FWF: P31889-N35}
\maketitle

\markboth{{\sf T. Hirschler and W. Woess}}
{{\sf Laplace equations}}
\baselineskip 15pt

\section{Introduction}\label{sec:intro}


A finite \emph{directed graph} is a finite set $X$ together with
a set $E \subset X^2$ of directed edges. Thus, we exclude multiple edges.
Loops, that is, edges of  the form $(x,x)$, play no role in our 
considerations and are excluded. We assume that $X$ is \emph{strongly
connected:} for any pair of points $x,y \in X$, there is a \emph{directed
path} from $x$ to $y$. By definition, it consists of vertices
$x=x_0\,, x_1\,,\dots, x_n=y$ such that $(x_{i-1}\,,x_i) \in E$ for 
$i = 1, \dots, n$, and $n$ is the \emph{length} of that path.

We now equip each directed edge with a \emph{weight} or \emph{conductance} $a(x,y) > 0$,
and speak of the resulting weighted graph as a \emph{directed network}. When $(x,y) \notin E$,
we set $a(x,y)=0$. We study the difference operator acting on functions
$u:X \to \C$ by
\begin{equation}\label{eq:Lap-a}
\Lp_a u(x) = \sum_y a(x,y)\bigl(u(y)-u(x)\bigr).
\end{equation}
Since nothing of what we are going to consider in this exposition depends substantially
on normalisation, we rather pass to 
\begin{equation}\label{eq:p}
p(x,y) = a(x,y)/m(x)\,,\quad \text{where}\quad m(x) = \sum_y a(x,y) 
\end{equation}
and the associated stochastic transition matrix $P= \bigl( p(x,y) \bigr)_{x,y \in X}$.
Then our normalised \emph{Laplace operator} is 
$$
\Delta = P - I\,,\quad \text{where} \quad Pu(x)= \sum_y p(x,y) u(y),
$$ 
and $I$ is of course the identity matrix over $X$, so that $Iu(x)=u(x)$.

We shall also designate a \emph{boundary} $\partial X$ of $X$. This is defined as
an arbitrary non-empty subset $\partial X \subset X$. It may be natural to require that for every 
$y \in \partial X$ there is $x\in X^o = X \setminus \partial X$ (the \emph{interior} of $X$)
such that $(x,y) \in E$, but mostly we will not need this. 
In the following, we let $L(X)$ be the linear space of
all functions $f: X \to \C$, thought of as column vectors.

Then we are interested in the solutions of the following types of basic problems:
\begin{itemize}
\item \emph{Poisson equation:} find $u \in L(X)$ such that $\Delta u = f$, where $f \in L(X)$
is given.
\item \emph{Neumann problem:} find $u \in L(X)$ such that $\Delta u = f$ on $X^o$, 
and the outer normal derivative satisfies $
\partial_{\vec n}u= g$ on $\partial X$,
where $f \in L(X^o)$ and 
$g \in L(\partial X)$ are given.
\item  \emph{Dirichlet problem:} find $u \in L(X)$ such that $\Delta u = f$ on $X^o$, 
and $u = g$ on $\partial X$,
where $f \in L(X^o)$ and $g \in L(\partial X)$ are given.
\item \emph{Mixed boundary problem:} decompose $\partial X = D \cuplus N$. Find 
$u \in L(X)$ such that $\Delta u = f$ on $X^o$, 
and $u = g{\res}_{D}$ on $D$ as well as $\partial_{\vec n}u = g{\res}_{N}$ on $N$,
where $f \in L(X^o)$ and $g \in L(\partial X)$ are given.
\item \emph{Robin boundary problem:} Find 
$u \in L(X)$ such that $\Delta u = f$ on $X^o$, 
and\\ $\alpha\cdot u + \beta \cdot \partial_{\vec n}u = g$ on $\partial X$, 
where $f \in L(X^o)$ and $g \in L(\partial X)$, as well as $\alpha, \beta \in \C \setminus \{0\}$ are given.
\item \emph{Schr\"odinger type equations:} in all of the above problems, replace 
$\Delta$ by $\Delta u - v\cdot u$, where $v \in L(X)$ is a suitable potential.
\end{itemize}

The first four of these problems are dealt with in \S \ref{sec:global}.
Above, one of the important
questions is how to define the outer normal derivative at a boundary point. We discuss
this with some care, also regarding sub-networks and their boundaries in $X$. 

The problems involving potentials are then considered in \S \ref{sec:potentials}. This includes
the Robin problem.

Subsequently, in \S \ref{sec:bi},  
we pass to the bi-Laplacian $\Delta^2$ and associated problems: 
\begin{itemize}
 \item the \emph{iterated Poisson equation,} 
\item the \emph{bi-Laplace Neumann problem,} and 
\item the \emph{bi-Laplace Dirichlet problem}
\end{itemize}
are as above, with $\Delta^2$ in the place of $\Delta$. 

Here, an interesting phenomenon comes up: in passing to $\Delta^2$, one gains only
one degree of freedom, independently of the size of the boundary. 
The general \emph{plate equation} in its first variant is to find $u \in L(X)$ such that 
$$
\Delta^2 u = f \; \text{ on }\; X^o\,,\quad \partial_{\vec n}u = g_1\; \text{ on }\;\partial X
\AND u = g_2\; \text{ on }\;\partial X,
$$
where $f \in L(X^o)$ and $g_1\,, g_2 \in L(\partial X)$ are given. It turns out
that only under a specific condition on these three functions, there is a solution.
This leads to a non-trivial bi-Laplace Dirichlet to Neumann map. 

The restrictive condition for solvability of the above plate equation induces us to
propose a second variant, which involves the \emph{second interior} of $X$
as well as a variant of the outer normal derivative.
Using this second interior, one can also solve an iterated Dirichlet problem
which can be extended to higher powers of the Laplacian.

In the final section \S \ref{sec:ex}, we compute in detail two examples.

The answers to all the problems considered here are obtained by 
suitable matrix operations.
They also have probabilistic
interpretations in terms of the underlying Markov chain. The latter is given by a sequence
$(Z_n)_{n \ge 0}$ of $X$-valued random variables with $\Prob[Z_{n+1} = y \mid Z_n = x] = p(x,y)$
for $x,y \in X$, where of course the Markov assumption holds, namely, conditionally upon
the value of $Z_n\,$, the past $(Z_j)_{0  \le j \le n-1}$ and the future $(Z_k)_{k \ge n+1}$
are independent. The unique stationary distribution of this Markov chain plays an important
role. This is one of the reasons why we have chosen to normalise the Laplacian.
All results can be easily restated for the operator $\Lp_a$ instead of $\Delta$. 
(For example, an equation of the form $\Lp_a u = f$ transforms into $\Delta u = \tilde f$,
where $\tilde f(x) = f(x)/m(x)$.)

We mention that with some care one can elaborate
the issues considered here also in the case when the graph $X$ is countably infinite, as long
as a stationary \emph{probability} measure $\pi$ exists. In probabilistic terms, this means
that the underlying Markov chain is \emph{positive recurrent.}

\smallskip

Motivation for studying the above issues may come from the discretisation of 
continuous PDEs, but the present note is not written in a spirit of numerical analysis,
and algorithmic features are not considered. Classical work is due to {\sc Duffin}~\cite{Du}, 
whose finite graphs are subsets of the rectangular lattice $\Z^d$, considered as 
discretisations of Euclidean domains; see e.g. also the very recent work of 
{\sc Varopoulos}~\cite{V1}, \cite{V2}. Another type of motivation comes
from ``electrical network'' theory, which has a long history and a well-known 
interplay with Markov chain theory, see the beautiful little book 
by {\sc Doyle and Snell}~\cite{DS}.

\smallskip

Several of the initial problems presented above are ``folklore'', such as the Dirichlet problem 
(which, however, keeps being ``rediscovered'', sometimes by complicated methods).
Basic results such as the solution of the Poisson equation for the Laplacian $P-I$
(in more generality than for finite state spaces) are part of the 
literature on the potential theory of Markov chains from the 1960ies, see
{\sc Kemeny and Snell}~\cite{KS} -- one of the most significant 
sources close to the spirit of the present note -- and the monograph by {\sc Kemeny, Snell and Knapp} 
(in particular, Chapter 9). However, it seems that this has remained secluded
from the non-probabilistic world of analysis and smooth potential
theory. More recent references from the discrete side are, for example, the lecture notes by 
{\sc Anandam}~\cite{An}, and, among other concerning Schr\"odinger type equations 
in the reversible (= self-adjoint) case,  
{\sc Bendito, Carmona and Encinas}~\cite {BCE0}, \cite{BCE1}, and, in particular \cite{BCE2}, as well as the work of these authors with {\sc Gesto}~\cite{BCEG}. 
See also the references in those papers.

Related only in part, there is a large body of work on discrete boundary value problems arising
for one-dimensional difference equations. Among the many references, we indicate the monographs 
by {\sc Atkinson}~\cite{At} and {\sc Agarwal}~\cite{Ag}. There is also a large number 
of contributions to the spectral theory of (self-adjoint) discrete Laplacians on finite 
networks, which however is not in the focus of the present work. 
See e.g. the books by {\sc Cvetkovi\'c, Doob and Sachs}~\cite{CDS} and {\sc Chung}~\cite{Ch},
and the beautiful article by {\sc Colin de Verdi\`ere}~\cite{CdV}.
Furthermore, there are some interesting
studies of Laplacians in discrete geometry, see e.g.  {\sc Kenyon}~\cite{Ke} or 
{\sc Bobenko and Springborn}~\cite{BoSp}. (These are only glimpses at the respective parts 
of the literature.)

On the other hand, the study of the bi-Laplacian in the discrete setting
has received only little and partly quite recent consideration in the literature; see e.g.
{\sc Yamasaki}~\cite{Ya}, {\sc Vanderbei}~\cite{Va}, {\sc Cohen et al.}~\cite{CCGS}, 
{\sc Anandam}~\cite{An}, {\sc Picardello and Woess}~\cite{PW}, or {\sc Hirschler and Woess}~\cite{HiWo}.
In the ``smooth'' literature, there is an ample body of work on this subject, and
we point at the lecture notes by {\sc Gazzola,  Grunau and Sweers}~\cite{GGS}.
In this context, the present paper may provide some new insight concerning the finite, 
discrete setting.

At several points, we have included discussions of possible approaches, where the discrete 
analogue of the smooth setting might allow different interpretations.
We have made an effort to present a comprehensive and coherent exposition, a task
which \emph{a posteriori} may appear easier than it was \emph{a priori}.

\section{The solutions of the basic problems}\label{sec:global}

While we write $L(X)$ for the space of functions $X \to \C$, of which we
think as column vectors, we consider measures as row vectors and write $M(X)$
for the resulting space. Of course, for $\mu \in M(X)$ and $A \subset X$,
we have $\mu(A) = \sum_{x \in A} \mu(x)$. In this sense, also $p(x,A) = \sum_{x \in A} p(x,y)$.

Strong connectedness of the graph $X$ amounts to \emph{irreducibility} 
of the non-negative matrix $P$, namely: for all $x,y \in X$ there is 
$n = n_{x,y}$ such that $p^{(n)}(x,y) > 0$, where
$$
P^n = \bigl(p^{(n)}(x,y)\bigr)_{x,y\in X}\,,\quad P^0=I
$$
(matrix powers). The following is very well known -- see e.g. {\sc Seneta}~\cite{Seneta} or 
{\sc Woess}~\cite{WMarkov}. For the sake of completeness we provide an outline of the proof.

\begin{lem}\label{lem:stationary} Irreducibility of $P$ implies the following.\\[3pt]
\emph{(a)} Every \emph{harmonic function} is constant,  that is, $h \in L(X)$ satisfies
$\Delta h = 0$ if and only if $h$ is constant.\\[3pt]
\emph{(b)}
There is a unique positive \emph{stationary probability distribution} $\pi \in M(X,\R)$, that is,
\begin{equation}\label{eq:stationary}
\pi\, P = \pi\,,\quad  \pi(x) > 0 \; \text{ for all } x \in X\,,\AND \sum_{x \in X} \pi(x) = 1. 
\end{equation}
\end{lem}

\begin{proof} For (a), if $h$ is harmonic then so are its real and imaginary part. Hence,
we may assume that $h$ is real, and  the statement follows from the minimum principle: let $x \in X$ such that $h(x) = \min h$.
Then
$$
\sum_y p^{(n)}(x,y)\underbrace{\bigl(h(y) - h(x)\bigr)}_{\displaystyle \ge 0} = 0\,,
$$ 
whence $h(y) = h(x)$ for each $x$ with $p^{(n)}(x) > 0$, and this holds for any $n$.
\\[3pt]
For (b), a straightforward compactness argument shows that there are stationary distributions,
namely, the accumulation points of the sequence 
$$
\frac{1}{n} \sum_{k=0}^{n-1} \mu\,P^k\,,
$$
where $\mu \in M(X)$ is an arbitrary non-vanishing non-negative measure.
If $\pi$ is stationary, then irreducibility implies $\pi(x)>0$ for all $x$, 
and  we can consider the new transition matrix 
\begin{equation}\label{eq:Phat}
\wh P = \bigl( \hat p(x,y)\bigr)_{x,y \in X} \quad \text{with} \quad 
\hat p(x,y) = \frac{\pi(y) p(y,x)}{\pi(x)}\,.
\end{equation} 
It is again irreducible, and 
$\mu P=\mu$ if and only if $h(x) = \mu(x)/\pi(x)$ satisfies $\wh P h = h$.
Thus, $h$ is constant by (a).
\end{proof}

For any two non-empty subsets $A, B \subset X$, we write 
$P_{A,B} = \bigl(p(x,y)\bigr)_{x\in A, y \in B}$ 
for the restriction of $P$ to $A \times B$. In particular, we write $P_A = P_{A,A}\,$, 
as well as $f_A$ for the restriction $f{\res}_{A}$ of $f \in L(X)$ to $A$. 
The following is also well-known.

\begin{lem}\label{lem:restrict} If $A \subset X$ strictly, then $I_A - P_A$
 is invertible, and 
$$
G_A:= (I_A - P_A)^{-1} = \sum_{n=0}^{\infty} P_A^n\,,
$$
a convergent series.
\end{lem}

\noindent\textit{Proof [Outline].} For each $x \in A$, let $n_x$ be the smallest $n$
such that $p^{(n)}(x,w) > 0$ for some $w \in X \setminus A$.
Then 
$$
\sum_{y \in A} p_A^{(n_x)}(x,y) < 1\,.
$$
We deduce that 
$$
\sum_{y \in A} p_A^{(n)}(x,y) < 1 \quad \text{for every }\; n \ge n_x\,.
$$
Let $N = \max \{ n_x: x \in A\}$. Then there is $\delta \in (0\,,\,1)$ such that
$$
\sum_{y \in A} p_A^{(N)}(x,y) \le 1-\delta \quad \text{for every }\; x \in A\,.
$$
Consequently, 
$$
\sum_{n=0}^{\infty} p_A^{(n)}(x,y) < \infty \quad \text{for all }\; x,y \in A. \eqno\square
$$

\medskip

While a priori, $P_A$ and $G_A$ are matrices over $A$, 
it will often be useful to consider them as matrices over the entire set $X$ with entries
$0$ outside of $A \times A$.

\smallskip

A comment on the probabilistic interpretation: for $x, y \in A$, the matrix entry 
$p_A^{(n)}(x,y)$ is the probability that the Markov chain starting at $x$ visits $y$ at 
time $n$, without ever leaving $A$  before that. Thus, the $(x,y)$-entry of the
matrix $G_A$ is the expected number of visits in $y$ of the Markov chain starting
at $x$ before it leaves the set $A$. Indeed, it is straightforward via the Borel-Cantelli
Lemma that it must leave $A$ with probability $1$ at some time.  

\smallskip

We choose a ``root'' vertex $o \in X$ as a reference point.
Let us say that a function $u \in L(X)$ is \emph{grounded,} if $u(o)=0$.
The following -- which can be seen as a consequence of the 
simple matrix version of the Fredholm alternative -- is a basis of 
almost everything which we are considering in this note. Therefore,
in spite of its simplicity, we call it a ``theorem''.
The measure $\pi$ is the one of Lemma \ref{lem:stationary}. Here and on several
occasions further below, we write weighted sums with respect to measures
on $X$, $X^o$ or $\partial X$ in the form of integrals in order to stress
the analogies with the continuous setting. Of course, we always remain in
the realm of finite-dimensional linear algebra.

\begin{thm}[Solution of the Poisson equation.] \label{thm:Poisson}
For $f \in L(X)$, the equation $\Delta u = f$ has a solution $u$ if and only if the \emph{charge}
$f$ is \emph{balanced,} that is,
\begin{equation}\label{eq:orthogonal}
\int_X f\, d\pi := \sum_{x \in X} \pi(x)\,f(x) = 0\,. 
\end{equation}
In this case, the unique grounded solution is given by
$$
u
= - G\ret_{X\setminus \{o\}}f\,,
$$
that is, 
$$
u(o)=0 \AND u(x) = -\sum_{y \in X\setminus \{o\}}\sum_{n=0}^{\infty} 
p_{X\setminus \{o\}}^{(n)}(x,y)f(y)\quad \text{for }\; x \in X\setminus \{o\}\,.
$$
All solutions are given by $u(\cdot) + c$, where $c \in \C$.
\end{thm}
 
\begin{proof} If $f$ is such that a solution $u$ exists, then (since $\pi P = \pi$)
$$
\int_X f\,d\pi = \pi (P-I)u = 
\int_X u\,d(\pi P) - \int_X u \,d\pi = 0\,,
$$ 
so that \eqref{eq:orthogonal} holds.

By Lemma \ref{lem:stationary}(a), the kernel of $\Delta$ consists of the constant
functions, whence it is one-dimensional. Thus, the dimension of the range of $\Delta$
is $|X|-1$. It is contained in the hyperplane $\{ f \in L(X) : \int_X f\,d\pi
= 0\}$,
which itself has dimension $|X|-1$. Thus, the range must be that hyperplane.
This proves the first part.

Now let $f$ satisfy \eqref{eq:orthogonal}. By Lemma \ref{lem:stationary}, any two solutions
of the associated Kirchhoff problem differ by a constant, so that there is a unique
grounded solution $u$. If $x \in X \setminus \{o\}$ then, since $u(o)=0$,
$$
f(x) = \Delta u(x) = \left(\sum_y p(x,y)u(y)\right)-u(x) 
= \left(\sum_{y \ne o}  p(x,y)u(y)\right)-u(x)\,,
$$
that is,
$$
f\ret_{X\setminus \{o\}} = \bigl(P\ret_{X\setminus \{o\}} - I\ret_{X\setminus \{o\}}\bigr)
u\ret_{\,X\setminus \{o\}}.
$$ 

Applying $G_{X\setminus \{o\}}$ to both sides from the left, the result follows via
Lemma \ref{lem:restrict}.
\end{proof}

We can equip $L(X)$ with the inner product $(f,g)_{\pi} = 
\int_X f\, \bar g\,d\pi$, so that
it becomes a finite-dimensional real Hilbert space, denoted $L(X,\pi)$. 
 
\begin{lem}\label{lem:adjoint} For all $f, g \in L(X)$,
$$
\int_X \bigl(f\,\Delta g - g \,\wh\Delta f\bigr) \,d\pi = 0\,, 
$$
where $\wh \Delta = I - \wh P$ with $\wh P$ given by \eqref{eq:Phat}.
\end{lem}

\begin{proof} This is immediate from the fact that $\wh\Delta$ is the 
 adjoint of $\Delta$ in $L(X,\pi)$.
\end{proof}

Given a boundary $\partial X$ as in the Introduction, note that there is no ``infinitesimal
change of measure'' from $\pi$ as the ``volume measure'' in $X^o$ to the same $\pi$ as
the ``surface measure'' on $\partial X$. Therefore it is most natural to define the
outer normal derivatives of $f \in L(X)$ on $\partial X$ by
\begin{equation}\label{eq:outer}
\partial_{\vec n} f(x) = - \Delta f(x)  \AND 
\wh \partial_{\vec n} f(x) = - \wh\Delta f(x)\,,\quad x \in \partial X\,.
\end{equation}
(As a matter of fact, this rather is ``minus the inner normal derivative''.) 
Then Lemma \ref{lem:adjoint} is equivalent with the following discrete, in general 
non-self-adjoint version of 
\emph{Green's second integral identity,} justifying our choice of the normal derivative.
\begin{equation}\label{eq:G2}
\int_{X^o} \bigl(f\,\Delta g - g \,\wh\Delta f\bigr) \,d\pi = 
\int_{\partial X} \bigl(f\,\partial_{\vec n}  g - g \,\wh\partial_{\vec n} f\bigr) \,d\pi
\,, \quad f,g \in L(X).
\end{equation}

 \begin{cor}[Solution of the Neumann problem.] \label{cor:Neumann}
The solution $u$ of $\Delta u = f$ on $X^o$ and $\partial_{\vec n}u = g$ on $\partial X$
 coincides with the solution of the Poisson equation $\Delta u = \tilde f$ on $X$,
where $\tilde f(x) = f(x)$ for $x \in X^o$ and $\tilde f(y) = -g(y)$ for $y \in \partial X$. 

In particular, the Neumann problem is solvable if and only if
$$
\int_{X^o} f \,d\pi = \int_{\partial X} g \,d\pi\,.
$$
\end{cor}

\begin{thm}[Solution of the Dirichlet problem.] \label{thm:Dirichlet}
For  $f \in L(X^o)$ and $g \in L(\partial X)$, there is a unique $u \in L(X)$ 
such that $\Delta u = f$ on $X^o$ and $u = g$ on $\partial X$. The solution
is given by
$$
u\ret_{\,X^o} = -G\ret_{X^o} ( f - P\ret_{X^o,\, \partial X}\,g)  \AND u\ret_{\,\partial X} = g.
$$
\end{thm}

\begin{proof}
For $x \in X^o$, we have
$$
\begin{aligned}
\Delta u(x) &= P\ret_{X^o}\,u\ret_{\,X^o}(x) - u(x)
+ P\ret_{X^o,\, \partial X}\,u\ret_{\,\partial X}(x)\\ 
&= f(x) - P\ret_{X^o,\, \partial X}\,g(x) + P\ret_{X^o,\, \partial X}\,u{\ret}_{\,{\partial X}}(x) = f(x).
\end{aligned}
$$
So, $u$ as given is indeed a solution. If $u_1\,,u_2$ are two solutions then 
$\Delta(u_1-u_2) = 0$ on $X^o$ and $u_1-u_2 \equiv 0$ on $\partial X$.
By the minimum principle, the real and complex parts of $u_1-u_2$ must assume their minimum and maximum on 
$\partial X$. Thus, 
$u_1 - u_2 \equiv 0$ on $X$. This shows uniqueness.

\smallskip

The above proof follows the classical potential theoretic line. Alternatively, one can also argue by direct operations; second proof:
we subdivide into blocks over $X^o$ and $\partial X$,
$$
\Delta = \begin{pmatrix} P\ret_{X^o} - I\ret_{X^o}& P\ret_{X^o,\partial X} \\
                         P\ret_{\partial X, X^o}& P\ret_{\partial X}-I\ret_{\partial X}
         \end{pmatrix}
\AND u = \begin{pmatrix} u_{\,\ret{X^o}}\\ g \end{pmatrix}.
$$
Then our equation can be written on $X^o$ as 
$$
(P\ret_{X^o} - I\ret_{X^o})u\ret_{\,X^o} + P\ret_{X^o,\partial X}\,g = f\,,
$$
and multiplying from the left with $G\ret_{X^o}$ yields the result.
\end{proof}

\begin{rmkdef}\label{rmk:Poisson}
For $x \in X^o$ and $y \in \partial X$, let $\nu_x(y) = G\ret_{X^o}P\ret_{X^o,\, \partial X}\uno_y(x)$
be the $(x,y)$-entry of the $X^o \times \partial X$-matrix $G\ret_{X^o}P\ret_{X^o,\, \partial X}$.
This is the probability that the Markov chain starting at $x$ enters the boundary
$\partial X$ at the point $y$. For $x \in \partial X$, we set $\nu_x= \delta_x\,$. Thus, $\nu_x= \nu_x^{\partial X}$ is a probability distribution on $\partial X$ for each $x \in X$. Recall that we can view $G\ret_{X^o}$ as a matrix over
$X \times X$ with entries $0$ outside of $X^o\times X^o$. 
Then we can write the solution of the Dirichlet problem as
\begin{equation}\label{eq:soldir}
u(x) = -G\ret_{X^o}f(x) + \int_{\partial X} g\,d\nu_x\,,\quad x \in X\,.
\end{equation}
Let $H(P, X^o)$ denote the linear space of \emph{harmonic functions on} $X^o$,
i.e., those $h \in L(X)$ which satisfy $\Delta h = 0$ on $X^o$. They are all of
the form $h(x) = \int_{\partial X} g\,d\nu_x\,$, where $g \in L(\partial X)$.
We define the \emph{exit} and the \emph{entrance boundary} as 
$$
\begin{aligned}
\partial_{\textrm{exit}}X &= \{ y \in \partial X : (x,y) \in E \;\text{for some}\; x \in X^o \} \AND\\
\partial_{\textrm{entrance}}X &= \{ z \in \partial X : (z,x) \in E \;\text{for some}\; x \in X^o \}\,. 
\end{aligned}
$$
If $x \in X^o$ then  $\nu_x(y) > 0$ implies that 
$y \in \partial_{\textrm{exit}}X$. If furthermore $P\ret_{X^o}$
is irreducible, that is, $X^o$ is strongly connected, then $\nu_x(y) > 0$ for all $(x,y)
\in X^o \times \partial_{\textrm{exit}} X$.

\smallskip

\emph{The Poisson equation is a special case of the Dirichlet problem.} Set $\partial X = \{ o \}$.
Given $f \in L(X^o)$, extend $f$ to $X$ by setting 
$$
f(o) = -\frac{1}{\pi(o)} \int_{X \setminus \{o\}} f\, d\pi\,.
$$
Then the solution of the Poisson equation grounded at $o$ coincides with the 
solution of the Dirichlet problem $\Delta u = f$ on $X \setminus \{o\}\,,\; u(o)=0\,$.
\end{rmkdef}

\begin{imp}[Dirichlet to Neumann.]\label{rmk:Dir-Neu}
In the smooth setting, the transformation from the Dirichlet 
to the Neumann problem is delicate, see e.g. {\sc Behrndt and Langer}~\cite{B-L}. 
In our discrete setting, it is straightforward:
typically, one considers the Dirichlet problem
$\Delta u = 0$ on $X^o$ and $u = g$ on $\partial X$, so that the solution 
is
$$
u(x) = \int_{\partial X} g\,d\nu_x\,,\quad x \in X.
$$
If we then set $g_1 = \partial_{\vec n}u = -\Delta u$ on  $\partial X$,
then $u$ solves the Neumann problem for the boundary function $g_1\,$. We have the linear
mapping $g \mapsto g_1 = (I\ret_{\partial X}-Q)g\,$, where 
$Q= P\ret_{\partial X} + P\ret_{\partial X,\, X^o} G\ret_{X^o} P\ret_{X^o,\,\partial X}\,$.
We shall examine this stochastic $\partial X \times \partial X$-matrix  in more detail in 
\S \ref{sec:bi}.  The kernel of that mapping
consists of the constant functions on $\partial X$.  
\end{imp}

The following is immediate from \eqref{eq:outer}
 and Theorem \ref{thm:Dirichlet}. 

\begin{lem}[Solution of the mixed boundary problem.] \label{lem:mixed}
Let $\partial X = D \cuplus N$ (both non-empty). For  
$f \in L(X^o)$ and $g \in L(\partial X)$, there is a unique
$u \in L(X)$ such that $\Delta u = f$ on $X^o$, 
and $u = g{\ret}_{\,D}$ on $D$ as well as $
\partial_{\vec n}u = g{\ret}_{\,N}$ on $N$. The solution is as follows:
let $\tilde f \in L(X^o \cup N)$ be given as
$$
\tilde f{\ret}_{{X^o}} 
= f \AND  \tilde f{\ret}_{N} = -g{\ret}_{\,N}\,.
$$
Then
$$
u{\ret}_{\,{X^o \cup N}} 
= -G{\ret}_{X^o\cup N} \bigl( \tilde f - P{\ret}_{X^o \cup N,\, D}\,g{\ret}_{\,D}\bigr)  
\AND u{\ret}_{\,D} = g{\ret}_{\,D}\,.
$$
\end{lem}

\begin{imp}[The self-adjoint case.] \label{rmk:reversible}
The most intensively studied case is the one when the edge set as well as the 
conductances are symmetric: $(x,y) \in E \iff (y,x)\in E$, and $a(x,y) = a(y,x)$. 
Then, with our normalisation at the root vertex $o$, the stationary distribution 
is given by $\pi(x) = m(x)/m(X)$, where $m(x)$ is as in \eqref{eq:p} and 
$m(X) = \sum_{y\in X} m(y)$. Thus,
$P$ is \emph{reversible,} that is, we have $\pi(x)p(x,y) = \pi(y)p(y,x)$ for all $x,y \in X$.
In this case, $\Delta= \wh \Delta$ is self-adjoint on $L(X,\pi)$, and Green's
identity \eqref{eq:G2} assumes the classical form. There is also the discrete
version of Green's first identity. This is present in a variety of 
textbooks, see e.g. {\sc Anandam}~\cite{An} or {\sc Jorgensen and Pearse}~\cite{JoPe}.
\end{imp}

\begin{dis}[Normal derivatives.] \label{dis:normal}
Given $P$ as above, we have solved the Poisson equation, using the associated stationary
distribution $\pi$. Then we have introduced the outer normal derivatives at $\partial X$
by \eqref{eq:outer}, with justification given by ``Green's identity'' \eqref{eq:G2}.

As we shall explain below, it may be more natural to define the outer normal derivatives
differently, that is, for $x \in \partial X$, we replace the original transition probabilities
$p(x,y)$ by new ones, $p'(x,y)$, which again are assumed to satisfy $\sum_y p'(x,y)=1$.
Then for $f \in L(\partial X)$,
\begin{equation}\label{eq:outer'}
\partial_{\vec n} f(x) = - \sum_y p'(x,y)\bigl(f(y) - f(x)\bigr), 
\end{equation}
and we can consider the Neumann and mixed boundary problems with this assignment.
Of course, this does not change the way how these problems are solved.
One just has to replace the original transition matrix $P$ by $P'$, 
whose entries are $p(x,y)$ if $x \in X$, and $p'(x,y)$, when $x \in \partial X$.

However, in this case the statement that the solutions of the Neumann and Poisson
problems coincide is no more valid. Namely, for the Poisson equation in Theorem 
\ref{thm:Poisson}, we use the stationary distribution $\pi$ of the matrix $P$,
while for the solution of the Neumann problem, it has to be replaced by
the stationary distribution $\pi'$ of the modified transition matrix, and in 
general $\pi \ne \pi'$. 

In particular, this becomes visible in the ``classical'' self-adjoint case
of \eqref{rmk:reversible}. There, the
reversibility property $\pi(x)p(x,y) = \pi(y)p(y,x)$ of $P$ is valid on all
of $X$, while for the modified matrix $P'$ it will in general fail if 
one of $x,y$ belongs to $\partial X$. Still, the change will affect the 
stationary distribution also in $X^o$, that is, in general 
$\pi'(x) \ne m(x)/m(o)$ even for $x \in X^o$.

\smallskip

Let us also give some additional motivation for the choice of the normal  derivatives in 
\eqref{eq:outer} and for the choice to consider also the non-self adjoint (non-reversible) 
case.
$$
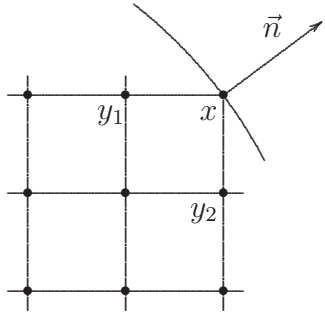
\beginpicture 

\setcoordinatesystem units <1.3mm,1.3mm>

\setplotarea x from 18 to 50, y from 5 to 43

\circulararc 15 degrees from 40 30 center at 0 0
\circulararc -9 degrees from 40 30 center at 0 0

\plot 40 30  18 30 /
\plot 42 20  18 20 /
\plot 42 10  18 10 /
\plot 40 30  40 8 /
\plot 30 32  30 8 /
\plot 20 32  20 8 /
\arrow <1.5mm> [.2,.6] from 40 30 to 50 37.5

\multiput  {$\scriptstyle \bullet$}  at 20 10  30 10  40 10
           20 20  30 20  40 20  20 30  30 30  40 30 / 

\put {$x$} [b] at 38.5 27.5   
\put {$y_1$} [b] at 28.5 27
\put {$y_2$} [b] at 38 17
\put {$\vec{n}$} [b] at 45 36

\endpicture
$$
Suppose we have a bounded open domain $D \subset \R^2$ with smooth boundary
curve $\partial D$. We suppose that we can discretise $D$ by inserting a finite square grid 
$X$, whose span is $h > 0$.
Let us also assume that $\partial D$ intersects the grid in a subset $\partial X$ of its vertex set,
so that at each vertex $x$ of $X^o = X \cap D$, its four neighbours $x \pm h\cdot e_1$
and $x \pm h\cdot e_2$ belong to the closure of $D$, where $e_1=(1,0)$ and $e_2 = (0,1)$. 
Then we can approximate the second derivatives by symmetric second order differences to
approximate the continuous Laplacian $\Delta_D$ at the vertices of $X^o$ as follows:
$$
\begin{aligned}
\Delta_D u(x) &\approx \frac{u(x+h\cdot e_1) - 2u(x) + f(x-h\cdot e_2)}{h^2}
+ \frac{u(x+h\cdot e_2) - 2u(x) + f(x-k\cdot e_2)}{h^2}\\
&= \sum_{y \sim x} a(x,y)\bigl(u(y)-u(x)\bigr) \quad \text{where}\quad
a(x,y) = 1/h^2 \;\text{ if }\; (x,y) \in E\,,\; x \in X^o\,.
\end{aligned} 
$$
The latter is an example of an operator $\Lp_a$ as defined in the Introduction.
The normalisation for passing to our ``stochastic'' version of $\Delta = \Delta_{[X]}$
means that we have to divide by $m(x)$, which has the constant value
$4/h^2$ on $X^o$, while we shall now discuss how to proceed at the boundary.
Let us look at a vertex $x \in \partial X$. We can approximate the \emph{inward}
pointing first order partial derivatives by first order differences, 
which are of the form 
$$
\frac{u(y) - u(x)}{h}\,, \quad \text{if }\; (x,y) \in E\,. 
$$
According to the slope of $\partial D$ at $x \in \partial X$ (which the grid
does not ``see''), the \emph{inner} normal derivative may be any convex 
combination of those partial derivatives, that is, of the form
$$
\sum_{y\,:\, (x,y) \in E} \lambda(x,y)\, \frac{u(y) - u(x)}{h}\,,
$$
where $\lambda(x,y) \ge 0$ and $\sum_y \lambda(x,y)=1\,$. For the outer normal
derivative, we just have to change sign.
We then should choose new conductances, which are not symmetric at 
the boundary vertices. 
Namely, for any directed edge $(x,y)$, we have
$$
a(x,y) = \begin{cases} 1/h^2\,,&\text{if}\; x \in X^o\\
                     \lambda(x,y)/h\,,&\text{if}\; x \in \partial X \,.       
       \end{cases}
$$ 
This leads to a very natural situation -- to be considered in
more generality than the above example -- where the edge 
set is symmetric,
and the conductances are symmetric along edges within $\partial X$,
but symmetry may fail for edges with an endpoint in $\partial X$. 
Nevertheless,  with the associated transition matrix $P'$ and 
stationary measure $\pi'$, the solution of the Neumann problem in
Corollary \ref{cor:Neumann} is valid for this very natural 
general choice of normal derivatives.

\smallskip

At this point we also mention old work -- in the perspective of numerical 
approximation -- of {\sc Bramble and Hubbard}~\cite{BH}
on the discretised Neumann problem on a domain $D \subset \R^2$, and its references.
\end{dis}

We extend this discussion of normal derivatives to the situation of sub-networks.

\begin{imp}[Remarks on sub-networks.]\label{rmk:sub}
Let $Y \subsetneq X$ be a connected sub-network. That is, as an induced (directed)  sub-graph of
$X$, it is strongly connected. In this case, we define the \emph{boundary} and \emph{interior}
of $Y$ as 
\begin{equation}\label{eq:sub}
\partial Y = \{ x \in Y : (x,z) \in E \;\text{for some}\;z \in X \setminus Y\}
\AND Y^o = Y \setminus \partial Y\,.
\end{equation}
In our discrete (non-infinitesimal) setting, we need to be careful when speaking about ``the'' Laplacian
on $Y$.

On one hand, we have the restricion $P_Y - I_Y$ of $\Delta = \Delta_{[X]}$ to $Y$.
On the other hand, if we start with the edge weights $a(x,y)$ as in \eqref{eq:Lap-a} as the initial
data, then $Y$ becomes a directed network on its own right, and the associated 
transition matrix $P_{[Y]}$ is given by
\begin{equation}\label{eq:pY}
p\ret_{[Y]}(x,y) = a(x,y)/m\ret_{[Y]}(x)\,,\quad \text{where} \quad m\ret_{[Y]}(x) = \sum_{w \in Y} a(x,w)\quad 
\text{for}\quad  x,y \in Y\,.
\end{equation}
In other words, the stochastic matrix $P_{[Y]}$ is obtained from the stubstochastic matrix $P_Y$ by
dividing each row by its row sum. Then we obtain the normalised Laplacian of $Y$ as 
$$
\Delta_{[Y]} = P_{[Y]} - I_Y\,.
$$
Accordingly, the natural choice for the associated outer normal derivative on $\partial Y$
is
\begin{equation}\label{eq:normalY}
\partial_{\vec n} f(y) = - \Delta_{[Y]} f(y)\,,\quad y \in \partial Y.
\end{equation}
We remark here that the restriction of $P_{[Y]}$ to $Y^o$ coincides with the restriction
of $P$ to $Y^o$. Therefore the Green kernel $G_{Y^o}$ is the same for $\Delta$ and $\Delta_{[Y]}\,$.
On the other hand, observe that the stationary distribution $\pi_{[Y]}$ of $P_{[Y]}$ is not
the restriction of $\pi$ to $Y$. (But in the reversible case, the two are proportional in 
$Y^o$.)

\smallskip

Coming back to the outer normal derivative, as mentioned, its definition \eqref{eq:outer} on $X$ 
is in reality the additive inverse of what one can consider as the \emph{inner} normal derivative.
This appears natural in several senses:
\begin{itemize} 
\item In the smooth setting,
the sum of the outer and the inner normal derivative is $0$, 
\item we have the 
discrete version \eqref{eq:G2} of  Green's second identity, and
\item the original network with is boundary does a priori not have an
\emph{exterior.}
\end{itemize}
But now, our sub-network $Y$ does have an exterior, namely $X \setminus Y$, 
and there is a second natural choice for the outer normal derivative on
$\partial Y$ of a function $f \in L(X)\,$, namely
\begin{equation}\label{eq:normalY*}
\partial_{\vec n}^* f(y) = \frac{1}{p(y,X \setminus Y)} \sum_{z \in X \setminus Y}
p(y,z)\bigl(f(z) - f(y)\bigr)\,.
\end{equation}
(The normalisation by $p(y,X \setminus Y)$ is convenient, but not crucial.)
There is no nice analogue of Green's second identity for this choice, so that we
mostly stick to \eqref{eq:normalY}, but in \S \ref{sec:bi} we shall see
an instance where \eqref{eq:normalY*} gains justification.
\end{imp}

In the Poisson equation $\Delta u = f$, the set $\{ x: f(x) \ne 0\}$ is the set
of inhomogeneities of that equation. If we are interested in the solution on the sub-network
$Y$ only, we want to shift the inhomogeneities from outside to $\partial Y$.
This is \emph{balayage,} which is also present in the 
old work \cite{KS}. Here, we display it in our setting.

\begin{imp}[Balayage.]\label{imp:bal}  Let $f \in L(X)$ with $\int_X f\,d\pi = 0$, and let $u$
be a solution of the Poisson equation $\Delta u = f$.
For a strict sub-network  $Y \subset X$, we consider the \emph{r\'eduite} (reduced function)
$u^Y$ of $u$ on $Y$. This is the solution of the Dirichlet problem
$$
\Delta u^Y = 0 \;\text{ on }\; X \setminus Y \AND u^Y = u   \;\text{ on }\; Y.
$$
Thus, we have $X \setminus Y$ in the role of $X^o$ and $Y$ in the role of $\partial X$,
and the solution is the function
$$
u^Y(x) = \int_{Y} u\,d\nu_x^{Y} \in H(P,X \setminus Y)\,;
$$
see Definition \& Remark \ref{rmk:Poisson}. We can also write $u^Y = u - v$, where
$v$ solves the Dirichlet problem
$$
\Delta v = f \;\text{ on }\; X \setminus Y \AND v = 0   \;\text{ on }\; Y\,, 
$$
so that $v = G\ret_{X \setminus Y} f$.
The \emph{balay\'ee} (swept out function) of the charge $f$ on $Y$ is then
$f^Y = \Delta u^Y$. 
We find
$$
f^Y(x) = \begin{cases} f(x)\,,&\text{if }\; x \in Y^o,\\
                       f(x) - P\ret_{\partial Y, X\setminus Y}G\ret_{X\setminus Y} \,f(x)\,,
                              &\text{if }\; x \in \partial Y,\\
                       0\,,&\text{if }\; x \in X \setminus Y\,.
         \end{cases}
$$ 
\end{imp}

\section{Equations including potentials}\label{sec:potentials}

We now extend the previous equations by adding a potential $v \in L(X)$.
The associated variant of the Poisson equation looks as follows: 
\begin{equation}\label{eq:KiSchroe}
\text{Given }\; f \in L(X)\,,\; \text{ find }\; u \in L(X)\;\text{ such that}
\quad \Delta u - v\cdot u = f \;\text{ on } X.
\end{equation}
Here, we shall assume that $v(x) \ge 0$ (real) for all $x$ and $v(x) > 0$ for some $x$.
More general variants will be discussed below.
We transform the above equation into
\begin{equation}\label{eq:transform}
\wt P u - u = \tilde f\,, \quad \text{where} \quad \tilde f(x) = \frac{f(x)}{1+v(x)}
\AND \tilde p(x,y) = \frac{p(x,y)}{1 + v(x)}\,.
\end{equation}   
If we want to reduce this to one of the problems of \S \ref{sec:global}, then we 
can add another state to $X$, setting $\wt X = X \cup \{\dag\}$, and extend
$\wt P$ to a stochastic matrix by setting
$$
\tilde p(x,\dag) = \frac{v(x)}{1+v(x)} \AND \tilde p(\dag,o)=1.
$$
The outgoing probabilities from $\dag$ are in reality irrelevant, and it might be more natural to
set $\tilde p(\dag,\dag)=1$, that is, at an $x \in X$, the original Markov chain ``dies''
and moves to the ``tomb'' state $\dag$ with probability $v(x)\big/\bigl(1+v(x)\bigr)$, where it 
remains forever. The above definition just serves to maintain an irreducible matrix $\wt P$. 
The associated Laplacian is now $\Delta_v = \wt P - I_{\wt X}$ on $\wt X$. 

Thus, if we define $\partial \wt X = \{\dag\}$ then \eqref{eq:transform} becomes 
the Dirichlet problem
$$
\Delta_v u =\tilde  f \;\text{ on }\; \wt X^o = X \AND u = 0 \;\text{ on }\; \partial \wt X =\{\dag\}. 
$$
At this point we return to considering $\wt P$ as a sub-stochastic matrix over $X$ alone.
Then we get from Lemma \ref{lem:restrict} that $I - \wt P$ is invertible on $X$,
and Theorem \ref{thm:Dirichlet} yields the following.

\begin{cor}[Poisson equation with potential.]\label{cor:KiSchroe}
With $\wt P$ and $\tilde f$ as in \eqref{eq:transform}, the unique solution of \eqref{eq:KiSchroe} is given on $X$ by 
$$
u = - \wt G \tilde f\,, \quad \text{where}\quad \wt G = (I - \wt P)^{-1} = \sum_{n=0}^{\infty} \wt P^n\,.
$$
\end{cor}

As a matter of fact, this can be seen as a special case of the following.

\begin{pro}[Dirichlet problem with potential.]\label{pro:DirSchroe} 
On $X = X^o \cuplus \partial X$ consider $f \in L(X^o)$ and $g \in L(\partial X)$. 
The unique solution $u$ of the problem 
$$
\quad \Delta u - v\cdot u = f \;\text{ on }\; X^0 \AND u=g\;\text{ on }\; \partial X  
$$
is given on $X^o$ by
$$
u{\ret}_{\,{X^o}} = -\wt G\ret_{X^o} ( \tilde f - \wt P\ret_{X^o,\, \partial X}\,g), 
$$
where $\tilde f$ on $X^0$ and $\wt P$ are as in \eqref{eq:transform} and 
$\wt G\ret_{X^o}$ is defined as in
Lemma \ref{lem:restrict} with respect to the matrix $\wt P\ret_{X^o}\,$.
\end{pro}
 
\begin{proof} This is immediate from the considerations preceding Corollary \ref{cor:KiSchroe},
by using the extended network $\wt X$ with enlarged boundary 
$\partial \wt X = \partial X \cup \{\dag\}$. Then we have the 
Dirichlet problem solved in Theorem \ref{thm:Dirichlet}, rewritten as
$$
\Delta_v u =\tilde  f \;\text{ on }\; \wt X^o = X^o \AND u = \tilde g \;\text{ on }\; \partial \wt X
$$
with boundary function $\tilde g$ given by $\tilde g{\ret}_{\,\partial X} = g$ and $\tilde g(\dag)=0$.
\end{proof} 

Recall that we have again the probabilistic interpretation of Definition \& Remark \ref{rmk:Poisson}:
$$
u(x) = -\wt G\ret_{X^o} \tilde f + \int_{\partial X} g\,d\tilde \nu_x\,,\quad x \in X^o\,,
$$
where $\tilde \nu_x(y)$ for $x \in X$ is the probability that the Markov chain on $\tilde X$
with the extended transition matrix $\wt P$ enters $\partial \wt X$ at $y$.
Note that $\tilde \nu_x$ is not necessarily a full probability measure, since
$\tilde\nu_x(\partial X) = 1 - \tilde\nu_x(\dag)$. As $\tilde g(\dag)=0$, the corresponding part of the 
boundary integral does not appear.
 
\smallskip

More generally, we may also consider a general complex valued potential $v \in L(X)$.
For our results, we need that 
\begin{equation}\label{eq:v}
 |1+v(x)| \ge 1 \; \text{ for all }\; x \AND |1+v(x)| > 1 \;\text{ for some }\; x \in X. 
\end{equation}
We can then define $\wt P$ and $\tilde f$ on $X$ as in \eqref{eq:transform}.
The matrix $\wt P$ is then in general no more sub-stochastic, since it may have negative
or complex entries, so that we lose the stochastic interpretation.
However,
$$
\sum_y |\tilde p(x,y)| \le 1  \; \text{ for all }\; x \AND
\sum_y |\tilde p(x,y)| < 1  \; \text{ for some }\; x \in X\,.
$$
Thus, using Lemma \ref{lem:restrict} once more,
we have on $X$ (without the additional point $\dag$) that
\begin{equation}\label{eq:tildeG}
\wt G = (I-\tilde P)^{-1} = \sum_{n=0}^{\infty} \wt P^n
\end{equation}
converges absolutely matrix element-wise, because $|\wt P^n| \le |\wt P|^n\,$.
After this observation, we see that Corollary \ref{cor:KiSchroe} and 
Proposition \ref{pro:DirSchroe} remain valid in this more general context.

\smallskip 

This comprises the ``classical'' Schr\"odinger equation, where $v(x)$ is
purely imaginary.

\smallskip

Let us now look at the Robin boundary problem. A bit more generally than in the
Introduction, we consider \emph{functions} $\alpha, \beta \in L(\partial X)$ and want
to  find  $u \in L(X)$ such that $\Delta u = f$ on $X^o$, 
and the boundary condition is 
\begin{equation}\label{eq:Robin}
\alpha(x)\,u(x) + \beta(x) \, 
\partial_{\vec n}u(x) = g(x) \quad \text{for all }\; x \in\partial X\,, 
\end{equation}
where $f \in L(X^o)$ and $g \in L(\partial X)$ are given.
We decompose $\partial X  = A \cuplus B$, where $B= \{ x \in \partial X: \beta(x) = 0\}$ 
(which may be empty) and
assume naturally that $\alpha(x) \ne 0$ for $x \in B$. 
Then we redefine the boundary and the interior: $\partial_{\beta}X = B$ and
$X_{\beta}^o = X^o \cuplus A$.

Now recall that $\partial_{\vec n}u(x) = -\Delta u(x)$ on $A$. Then we can rewrite the 
Robin problem as follows. Find $u \in L(X)$ such that
\begin{equation}\label{eq:Robin-a}
\begin{aligned}
\Delta u - v\cdot u &= f_{\beta}  \;\text{ on }\; X_{\beta}^o 
\AND u = g_{\alpha}  \;\text{ on }\; \partial_{\beta}X\,, \quad\text{where}\\[3pt]
v(x) &= \begin{cases}
       0&\text{for } x \in X^o\\ \alpha(x)/\beta(x)&\text{for } \; x \in A
       \end{cases},\\[3pt]
f_{\beta}(x) &= \begin{cases}
       f(x)&\text{for }\; x \in X^o\\-g(x)/\beta(x)&\text{for } \; x \in A
       \end{cases}, \AND\\[3pt]
g_{\alpha}(x) &= g(x)/\alpha(x)\;\text{ for}\; x \in B\,.
\end{aligned} 
\end{equation}

\begin{cor}[Solution of the Robin boundary problem.]\label{cor:Robin}
 Suppose that  $|\alpha(x) + \beta(x)| \ge |\beta(x)|$  for all $x \in B$, and that the inequality
is strict for some $x \in B$. Then the solution of the problem with boundary
condition \eqref{eq:Robin} is given on $X^o_{\beta} = X^o \uplus A$ by 
Proposition \ref{pro:DirSchroe} after replacing $X^o$ by $X^o_{\beta}$
and $\partial X$ by $\partial_{\beta}X =B$, as well as $g$ by $g_{\alpha} = g/\alpha\,$, 
and setting
$$
\tilde p(x,y) = \begin{cases}
                p(x,y)&\text{for }\; x \in X^o\\
                \frac{\beta(x)}{\alpha(x)+\beta(x)}p(x,y)&\text{for }\; x \in A 
                \end{cases}
\AND \tilde f(x) = \begin{cases}
                f(x)&\text{for }\; x \in X^o\\
                \frac{-1}{\alpha(x)+\beta(x)}g(x)&\text{for }\; x \in A 
                \end{cases}.
$$
\end{cor}
In the case when $B=\emptyset$, this reduces of course to the solution given in Corollary
\ref{cor:KiSchroe}.

\section{The bi-Laplacian}\label{sec:bi}
 
Without specifying a boundary, the first natural choice for the \emph{bi-Laplacian} is 
$\Delta^2$, the square of the matrix $\Delta$. 
In the associated Poisson equation one does not gain a degree of
freedom with respect to the simple Laplacian.

\begin{thm}[Iterated Poisson equation.] \label{thm:Poisson2}
For $f \in L(X)$, the equation $\Delta^2 u = f$ has a solution $u$ if and only if
$\int_X f\, d\pi = 0\,$. 
In this case, the unique solution grounded at $o$ is given by
$$
u = G\ret_{X\setminus \{o\}}\Bigl(G\ret_{X\setminus \{o\}}f
                                  - \int_X G\ret_{X\setminus \{o\}}f\, d\pi\Bigr)\,.
$$
All solutions are given by $u(\cdot) + c$, where $c \in \C$.
\end{thm}

\begin{proof} Set $\Delta u = v$. Then $v$ must be a solution of the Poisson equation
 of Theorem \ref{thm:Poisson}. That is, 
$$
v =  -G\ret_{X\setminus \{o\}}f + c\,,\quad \text{where }\; c \in \C.
$$
Next, $u$ must solve the Poisson equation $\Delta u = v$. Thus, we must have
$\int_X v\,d\pi =0$. This forces $c = \int_X G\ret_{X\setminus \{o\}}f\, d\pi\,$,
so that we get the proposed grounded solution.
\end{proof}

In particular, the bi-harmonic functions on $X$, i.e., the elements of the kernel of $\Delta^2$, are
the constant functions.

We now consider boundary value problems for the bi-Laplacian. Recall from 
Definition \& Remark \ref{rmk:Poisson}  the family of probability measures $\nu_x = \nu_x^{\partial X}$ ($x \in X$) 
on $\partial X$. Given the stationary distribution $\pi$, we define $\nu_{\pi}$ by
$$
\nu_{\pi} = \sum_{x \in X} \pi(x)\,\nu_x = \pi\, G\ret_{X^o}P\ret_{X^o,\,\partial X}\,.
$$
\begin{thm}[Bi-Laplace Neumann problem.] \label{thm:BiNeumann}
Let $f \in L(X^o)$ and $g \in L(\partial X)$. Then the boundary value problem
$$
\Delta^2 u = f \; \text{ on }\; X^o \AND \partial_{\vec n}u = g\; \text{ on }\;\partial X
$$
is solvable if and only if 
\begin{equation}\label{eq:condition}
\int_{X^o} G\ret_{X^o} f\, d\pi + \int_{\partial X} g\, d\nu_{\pi} = 0\,.
\end{equation}
In this case, all solutions are given by
$$
G\ret_{X \setminus \{o\}}G\ret_{X^o}f + G\ret_{X \setminus \{o\}} h + c\,, \quad \text{where} \quad h(x) = \int_{\partial X} g\, d\nu_x \AND c \in \C\,.
$$
\end{thm}

\begin{proof}
Recall from \eqref{eq:outer} that $\partial_{\vec n}u = - \Delta u$ on $\partial X$.
Hence, if we set $v = \Delta u$ then $v$ must solve the Dirichlet problem
$$
\Delta v = f \;\text{ on }\; X^o \AND v = -g \;\text{ on }\; \partial X\,.
$$
By Theorem \ref{thm:Dirichlet} and Remark \ref{rmk:Poisson}, the unique solution is
$v = -G\ret_{X^o} f - h\,$, where the harmonic function $h \in H(P,X^o)$ 
is as stated.
Since $v = \Delta u$, we must necessarily have $\int_X v\,d\pi = 0$, that is,
\eqref{eq:condition} must hold. In this case, $u$ solves the Poisson equation
$\Delta u = v$, which leads to the proposed solution(s).
\end{proof}

Before passing to the bi-Laplace Dirichlet problem, we need some 
preparations, involving the two matrices
\begin{equation}\label{eq:RS}
R = P\ret_{\partial X,\, X^o} G\ret_{X^o}^{\,\,\,2} P\ret_{X^o,\,\partial X} 
\AND
S = (I\ret_{X^o} - P\ret_{X^o})^2 + P\ret_{X^o,\,\partial X} P\ret_{\partial X,\, X^o}
\end{equation}
over $\partial X$ and $X^o$, respectively. We shall see below that invertibility of these matrices
is important in the context of bi-Laplace boundary value problems.

\begin{lem}\label{lem:invertible} The matrix 
$S$ is invertible  $\iff I_{\partial X} + R$ is invertible.
\end{lem}

\begin{proof}
We have the following two identities.
$$
P\ret_{\partial X,\, X^o} G\ret_{X^o}^{\,\,\,2} S 
= \bigl(I\ret_{\partial X} + R\bigr)P\ret_{\partial X,\, X^o}
\AND S G\ret_{X^o}^{\,\,\,2} P\ret_{X^o,\,\partial X}  
= P\ret_{X^o,\,\partial X}\bigl(I\ret_{\partial X} + R\bigr). 
$$
If $\bigl(I_{\partial X} + R\bigr)$ is invertible then we can multiply the first identity by  
$\bigl(I_{\partial X} + R\bigr)^{-1}$ from the left and expand
$$
G\ret_{X^o}^{\,\,\,2} S = I\ret_{X^o} + 
\Bigl(G\ret_{X^o}^{\,\,\,2}P\ret_{X^o,\,\partial X}\bigl(I\ret_{\partial X} + R\bigr)^{-1}
P\ret_{\partial X,\, X^o}\Bigr) G\ret_{X^o}^{\,\,\,2} S\,,
$$
which implies that $G\ret_{X^o}^{\,\,\,2} S$ and hence also $S$ are invertible.

Analogously, if $S$ is invertible then we can multiply the second identity by $S^{-1}$ from the right
and expand
$$
I\ret_{\partial X} + R = I\ret_{\partial X} + 
\Bigl( P\ret_{\partial X,\, X^o} S^{-1}P\ret_{X^o,\,\partial X}  \Bigr) 
\bigl(I\ret_{\partial X} + R\bigr)\,,
$$
which implies that $I\ret_{\partial X} + R$ is invertible.
\end{proof}

Let 
$\varPi = \diag\bigl(\pi(x)\bigr)_{x \in X}$ and $\varPi_{X^o}$ and $\varPi_{\partial X}$
the respective  restrictions of this diagonal matrix the the interior and the boundary of $X$.
Recall once more from Definition \& Remark \ref{rmk:Poisson} the hitting distributions 
$$
\nu_w(x) = \Prob [ \taub < \infty\,,\; Z_{\taub} = x \mid Z_0 =w ]\,,\quad w \in X^0\,,\; 
x \in \partial X
$$ 
of our Markov chain on the boundary. Finally, also recall from \eqref{eq:Phat} the
reversed transition matrix $\wh P = \varPi^{-1}P^t\,\varPi$. It is irreducible along with $P$,
and there are the associated hitting distributions $\hat \nu_w\,,\; w \in X^o\,$, on $\partial X$.
If $\hat\nu_w(x) > 0$ for some $w \in X^o$, then $x \in \partial_{\text{entrance}}X$, the 
entrance boundary of the original Markov chain of Definition \& Remark \ref{rmk:Poisson}. 

\begin{lem}\label{lem:invert}
Let $\;\varUpsilon= \bigl( \nu_w(x)\bigr)_{w \in X^o,\, x \in \partial X}$ and 
 $\;\wh \varUpsilon= \bigl( \hat \nu_w(x)\bigr)_{w \in X^o,\, x \in \partial X}\,$, two matrices
over $X^o \times \partial X$. Then the matrix $R$ satisfies
$$
\varPi\ret_{\,\partial X}R = \wh \varUpsilon^{\,T} \varPi\ret_{\,X^o} \varUpsilon.
$$
In particular, if $P$ is reversible, i.e. $\wh P = P$, 
so that $\Delta$ is self-adjoint on $L(X,\pi)$, 
then the matrix $I_{\partial X} + R$ is invertible.
\end{lem}

\begin{proof}
We have  $G\ret_{X^o} P\ret_{X^o,\,\partial X} = \varUpsilon$, see Remark \ref{rmk:Poisson}.
Therefore $R = P\ret_{\partial X,\, X^o} G\ret_{X^o}\varUpsilon$. Furthermore, by the definition
of $\wh P$, we have $\varPi P = \wh P^T \varPi$.  Therefore 
\begin{equation}\label{eq:NT}
\varPi\ret_{\,\partial X} P\ret_{\partial X,\, X^o} G\ret_{\,X^o} 
= \bigl(\wh G\ret_{X^o} \wh P \ret_{\partial X,\, X^o}\bigr)^T \varPi\ret_{\,X^o} 
= \wh \varUpsilon^T \varPi\ret_{\,X^o}\,.
\end{equation}
This proves the formula for $\varPi\ret_{\,\partial X}R$. Now let $g_1\,, g_2 \in L(X)$. Set
$$
 \hat h_1(w) = \int_{\partial X} g_1\,d\hat\nu_w \AND h_2(w) = \int_{\partial X} g_2\,d\nu_w\,, \quad
 w \in \partial X\,.
$$
Then \eqref{eq:NT} implies that
\begin{equation}\label{eq:innprod}
(g_1\,,R  g_2)_{\pi,\partial X} = (\hat h_1\,, h_2)_{\pi, X^o}\,, 
\end{equation}
where the subscripts indicate that the inner products are taken with respect to
the restriction of $\pi$ to $\partial X$, resp. $X^o$.

Now suppose that $P$ is reversible. Then $\hat \nu_w = \nu_w$ for all $w \in X^o$, 
whence $\wh \varUpsilon= \varUpsilon$. 
Assume that $R$ has a real eigenvalue $-\alpha < 0$ with associated 
non-zero eigenfunction $g$. Set $g_1=g_2=g$. Then also $\hat h_1=h_2=:h$. We get
$$
0 >  -\alpha (g, g)_{\pi,\partial X} = (g,R g)_{\pi,\partial X} = (h\,, h)_{\pi, X^o} \ge 0,
$$
a contradiction. Therefore, in the reversible case, 
$\alpha \cdot I_{\partial X} + R$ is invertible for 
every $\alpha > 0$, in particular, for $\alpha = 1$.
\end{proof}

Note that we can factorise 
$$
S = (I\ret_{X^o} - P\ret_{X^o})(I\ret_{X^o} - P\ret_{X^o} + \varUpsilon P\ret_{\partial X,\, X^o}).
$$
Thus, $S$ is invertible if and only if the second of those factors is invertible, and
then we can write
\begin{equation}\label{eq:SKG}
S^{-1} = K\ret_{X^o}G\ret_{X^o}\,, \quad\text{where}\quad 
K\ret_{X^o} = (I\ret_{X^o} - P\ret_{X^o} + \varUpsilon P\ret_{\partial X,\, X^o})^{-1}.
\end{equation}

We mention at this point that so far, we did not find a 
general condition beyond reversibility which guarantees invertibility of 
$S$, resp. $I\ret_{\partial X} + R$.

\begin{thm}[Bi-Laplace Dirichlet problem.] \label{thm:BiDirichlet}
Let $f \in L(X^o)$ and $g \in L(\partial X)$. If the matrix $S$ of \eqref{eq:RS} is  invertible -- in particular, in the reversible case -- the boundary value problem
$$
\Delta^2 u = f \; \text{ on }\; X^o \AND u = g\; \text{ on }\;\partial X
$$
has a unique solution. It is given by
$$
u =  K\ret_{X^o}G\ret_{X^o}(f + Ug) \quad \text{on }\; X^o\,,
$$
where the $X^o \times \partial X$-matrix $U$ is given by
$$
U = (I\ret_{X^o} - P\ret_{X^o})P\ret_{X^o,\,\partial X} 
+ P\ret_{X^o,\,\partial X}(I\ret_{\partial X} - P\ret_{\partial X})\,.
$$
\end{thm}

\begin{proof}
We use once more the block decomposition
$$
u = \begin{pmatrix} u_{\,\ret{X^o}}\\ g \end{pmatrix}\,,\quad 
\Delta = \begin{pmatrix} P\ret_{X^o} - I\ret_{X^o}& P\ret_{X^o,\partial X} \\
                         P\ret_{\partial X, X^o}& P\ret_{\partial X}-I\ret_{\partial X}
         \end{pmatrix}
\AND
\Delta^2 = \begin{pmatrix} S & -U \\ -U' & S'
           \end{pmatrix}\,,
$$
where $S' = (I\ret_{\partial X}-P\ret_{\partial X})^2 + P\ret_{\partial X, X^o}P\ret_{X^o,\partial X}$
and $U' = (I\ret_{\partial X}-P\ret_{\partial X})P\ret_{\partial X, X^o} + 
(I\ret_{\partial X}-P\ret_{\partial X})P\ret_{\partial X, X^o}\,$.
Thus, the equation for $u$ on $X^o$ becomes
$$
(\Delta^2 u)\ret_{\,X^o} = S u\ret_{\,X^o} - Ug\,, 
$$
which is equal to $f$ precisely when $u\ret_{\,X^o}$ has the proposed form.
\end{proof}

While in the ``classical'' reversible case, the matrices $S$ and $I\ret_{\partial X} + R$ are
always invertible, so that the above applies, this is not true in general.

\begin{exa}\label{ex:noninv}
Let $X = \{ 0, 1, \dots, 2N-1\}$ and $P$ be given by $p(k,k+1)=1$, where $k+1$ is taken modulo
$2N$. All other transition probabilities are $=0$. Thus, the associated graph is an oriented
circle of length $2N$. Now we take $X^o = \{ 0, 2, 4, \dots, 2N-2\}$ and 
$\partial X = \{ 1, 3, 5,\dots, 2N-1\}$. Then $S$ is the $N \times N$-matrix
$$
S = \begin{pmatrix}
     1 & 1 &        &        & \\
       & 1 & 1      &        & \\
       &   & \ddots & \ddots & \\
       &   &        & 1      & 1 \\
     1 &   &        &        & 1 
    \end{pmatrix},
$$
which is not regular when $N$ is even.
\end{exa}

\begin{rmk}\label{rmk:inv}
Recall from Definition \& Remark \ref{rmk:Poisson}  the space $H(P,X^o)$
of functions which are harmonic on $X^o$.
One can transform  \eqref{eq:innprod} into the following.

\smallskip

The matrices $S$ and equivalently, $I+R$ are \emph{not} invertible
($-1$ is an eigenvalue of $R$)

$\iff\;$ there is $h_2 \in H(P,X^o), \; h_2 \ne \zero,$ such that $(\hat h_1,h_2)_{\pi} =0$
for every $\hat h_1 \in H(\wh P,X^o)$

$\iff\;$ there is $\hat h_1 \in H(\wh P,X^o), \; \hat h_1 \ne \zero,$ such that $(\hat h_1,h_2)_{\pi} =0$ for every $h_2 \in H(P,X^o)$

$\iff\;$ there are non-zero solutions $u$ of the bi-Laplace Dirichlet problem 
$$
\Delta^2 u = 0 \; \text{ on }\; X^o \AND u = 0\; \text{ on }\;\partial X.
$$
One sees once more that $I\ret_{\partial X}+R$ and $S$ are invertible in the reversible case. 
So far, besides Example \eqref{ex:noninv}, we did not find further (non-reversible) examples where they are non-invertible. 
\end{rmk}

\begin{rmk}\label{rmk:not}
We see that in general we cannot expect to get solutions of the problem to 
find $u \in L(X)$ such that
\begin{equation}\label{eq:plate}
\Delta^2 u = f \; \text{ on }\; X^o\,,\quad \partial_{\vec n}u = g_1\; \text{ on }\;\partial X
\AND u = g_2\; \text{ on }\;\partial X
\end{equation}
for arbitrary $f \in L(X^o)$ and $g_1\,, g_2 \in L(\partial X)$. Indeed, first of all,
$f$ and $g=g_1$ have to satisfy the conservation law \eqref{eq:condition}, and after that,
we have only one degree of freedom left for the choice of $g_2$. That is, given $g_1\,$, we can choose 
$g_2(z)$ for precisely one element $z \in \partial X$. Inserting this value into
the solution of the bi-Laplace Neumann problem, we find the constant $c$, after which the
other values of $g_2$ on $\partial X$ are determined.

\smallskip 

The only case where this is completely satisfactory is 
the one where $\partial X = \{z\}$ consists of
one point only.  In this case, $g_1$ and $g_2$ are two \emph{constants,} and $\nu_x(z) = 1$
for each $x \in X$, so that condition \eqref{eq:condition} becomes
$$
\int_{X \setminus \{z\}} G\ret_{X \setminus \{z\}} f\, d\pi + g_1 = 0\,.
$$
If this holds then there is a unique solution to equation \eqref{eq:plate} 
for $X^o=X \setminus \{z\}$.
Since the choice of the root $o$ for the solution of the Poisson equation in Theorem 
\eqref{thm:Poisson} was arbitrary, we may as well use $o=z$ in this case, and 
then
$$
u = G\ret_{X \setminus \{z\}}^{\,\,\,2} f + g_1 \cdot G\ret_{X \setminus \{z\}} \uno + g_2\,,
$$
where $\uno$ is the constant function with value $1$.
\end{rmk}

In spite of what was said in the last remark, it will turn out to 
be of great interest to derive directly which conditions have to be fulfilled by 
$g_1$ and $g_2$ (as well as $f$) so that
\eqref{eq:plate} can be solved. This 
will be enhanced by 
the probabilistic interpretation (which, however, may be skipped by readers 
who prefer to avoid probability).

\begin{dfn}\label{def:bdryproc}
Consider the stopping time $\taub = \inf \{ n \ge 1: Z_n \in \partial X \}$. 
The \emph{boundary Markov chain} is defined by the transition matrix
$$
Q = \bigl(q(x,y)\bigr)_{x,y \in \partial X}\,,\quad q(x,y) 
= \Prob[\taub < \infty\,,\;Z_{\taub}=y \mid Z_0 = x]. 
$$ 
The \emph{boundary Laplacian} is $\Delta\ret_{\,\,\partial X} = Q - I\ret_{\,\partial X}\,$. 
\end{dfn}

This gives rise to the induced Markov chain on $\partial X$, i.e., the original 
Markov chain observed at the successive 
visits to $\partial X$. The following is a well known consequence 
of finiteness of $X$ and irreducibility of $P$,
see e.g. \cite[\S 6.C]{WMarkov}.

\begin{lem}\label{lem:bdryproc}
The stopping time $\taub$ is almost surely finite for any starting point, 
the matrix $Q$ is stochastic and irreducible, and
$$
Q = P\ret_{\partial X} + P\ret_{\partial X,\, X^o} G\ret_{X^o} P\ret_{X^o,\,\partial X}\,. 
$$ 
Furthermore, denoting by $\pi\ret_{\,\partial X}$ the restriction of the stationary distribution
$\pi$ of $P$ to $\partial X$, we have that $\pi\ret_{\,\partial X}$ is stationary for $Q$,
that is, $\pi\ret_{\,\partial X}Q = \pi\ret_{\,\partial X}\,$.
\end{lem}

The last identity can be obtained by direct matrix operations.
We insert a small observation concerning the matrix $R$ over $\partial X$ 
of \eqref{eq:RS}. Recall that $G\ret_{X^o} = \bigl(I\ret_{X^o} - P\ret_{X^o}\bigr)^{-1}$, and consider more
generally
$$
G\ret_{X^o}(\lambda) = \bigl(\lambda\cdot I\ret_{X^o} - P\ret_{X^o}\bigr)^{-1}
\AND Q(\lambda) = P\ret_{\partial X} 
+ P\ret_{\partial X,\, X^o} G\ret_{X^o}(\lambda) P\ret_{X^o,\,\partial X}\,.
$$
It is well-defined and analytic in a neighbourhood of $\lambda =1$, and as a consequence
of the well-known resolvent equation 
$G\ret_{X^o}(\lambda_1) - G\ret_{X^o}(\lambda_2) = (\lambda_2-\lambda_1) 
\cdot G\ret_{X^o}(\lambda_1) G\ret_{X^o}(\lambda_2)$, we get
\begin{equation}\label{eq:derive}
\frac{d}{d\lambda}\bigl(\lambda \cdot I\ret_{\partial X} - Q(\lambda)\bigr) =
I\ret_{\partial X} + R(\lambda)\,,\quad \text{where}\quad
R(\lambda) = P\ret_{\partial X,\, X^o} G\ret_{X^o}(\lambda)^2 P\ret_{X^o,\,\partial X}\,.
\end{equation}
We have $Q = Q(1)$ and  $R = R(1)$, so that in the following, we may interpret 
$-(I_{\partial X} + R)$ as the derivative (at $\lambda = 1$) of $\Delta_{\partial X}\,$.

\begin{thm}[Discrete plate equation, Variant 1.]\label{thm:plate}
For  $f \in L(X^o)$ and $g_1\,, g_2 \in L(\partial X)$, the problem \eqref{eq:plate} admits
a solution if and only if
\begin{equation}\label{eq:plate-cond}
\Delta\ret_{\,\,\partial X} g_2 + P\ret_{\partial X,\, X^o} G\ret_{\,X^o}^{\,\,2} f
= -(I\ret_{\,\partial X} + R)g_1\,,
\end{equation}
where $\Delta\ret_{\,\,\partial X} = Q - I\ret_{\,\partial X}$ is the boundary Laplacian,
and $R$ is given by \eqref{eq:RS}.

In this case, the solution is given by 
$$
u = G\ret_{\,X^o}^{\,\,2} f + G\ret_{\,X^o} h_1 + h_2\,, \quad \text{where}\quad
h_i(x) = \int_{\partial X} g_i\,d\nu_x \;\; (x \in X, \; i=1,2).
$$
Equivalently, for any choice of $z \in \partial X$, 
$$
u= G\ret_{X \setminus \{z\}}G\ret_{X^o}f + G\ret_{X \setminus \{z\}} h_1 + g_2(z).
$$
\end{thm}

\begin{proof}
Since we require $\Delta^2 u = f$ on $X^o$, the function $v = \Delta u$ must be
given on all of $X$.  The function $v$ must solve the Dirichlet problem
$$
\Delta v = f \;\text{ on }\; X^o \AND v = -g_1 \;\text{ on }\; \partial X\,.
$$
By \eqref{eq:soldir}, this yields 
$$
v = -G\ret_{\,X^o} f - h_1\,.
$$
Then $u$ must solve the Dirichlet problem
$$
\Delta u = v{\ret}_{\,X^o} \;\text{ on }\; X^o \AND u = g_2 \;\text{ on }\; \partial X\,,
$$
whence
$$
u = -G\ret_{\,X^o}v{\ret}_{\,X^o} + h_2\,.
$$
Thus, we get the first formula for the proposed solution, but we still need 
to check compatibility.
In $x \in \partial X$ then we must have $\Delta u(x) = v(x) = -g_1(x)$, that is
$$
\begin{aligned}
g_1(x) &= u(x) - \sum_y p(x,y)u(y)\\
&= g_2(x) - \sum_{y \in \partial X} p(x,y)g_2(y) 
- \sum_{w \in X^o} p(x,w)\Bigl( h_2(w) + G\ret_{\,X^o}^{\,\,2} f(w) + G\ret_{\,X^o} h_1(w)\Bigr)
\end{aligned}
$$
Now recall that we consider functions as column vectors, and that 
$h_i = \varUpsilon g_i = G\ret_{X^o} P\ret_{X^o,\,\partial X}g_i$. Thus, the above means that
$$
g_1 = g_2 - P\ret_{\partial X} g_2 - P\ret_{\partial X,X^o}G\ret_{X^o} P\ret_{X^o,\,\partial X}g_2
-P\ret_{\partial X,X^o}G\ret_{\,X^o}^{\,\,2} f 
- P\ret_{\partial X,X^o}G\ret_{X^o}^{\,\,2} P\ret_{X^o,\,\partial X}g_1\,.
$$
Reordering the terms, we obtain that condition \eqref{eq:plate-cond} is necessary for the solution. 
If \eqref{eq:plate-cond} holds then we can read the above arguments backwards and see that
indeed $\Delta u = v$ on the whole of $X$, as required, so that the solution is feasible.

The second formula for the solution now follows from  Theorem \ref{thm:BiNeumann}.
\end{proof}

A discussion of condition \eqref{eq:plate-cond} is now in place.

\begin{imp}[Dirichlet to Neumann for the Bi-Laplacian.]\label{rmk:Dir-Neu-Bi}
We fist observe that the solution of \eqref{eq:plate} is also a solution of
the bi-Laplace Neumann problem, so that $f$ and $g_1$ must satisfy \eqref{eq:condition}.
We may ask where this is ``hidden'' in the condition \eqref{eq:plate-cond}. 
Given $f$ and $g_1\,$, the latter is a Poisson equation for the boundary 
Laplacian for the determination of $g_2\,$. Since the transition matrix $Q$
is irreducible with invariant measure $\pi\ret_{\partial X}$, Theorem \ref{thm:Poisson}
implies that for admitting solution, it is necessary and sufficient that
\begin{equation}\label{eq:condition'}
\int_{\partial X} \Bigl( P\ret_{\partial X,\, X^o} G\ret_{\,X^o}^{\,\,2} f
+ (I\ret_{\,\partial X} + R)g_1\Bigr) \,d\pi = 0.
\end{equation}
With some small effort, this transforms precisely into \eqref{eq:condition}.
Indeed, \eqref{eq:NT} implies that 
$$
\pi(x) P\ret_{\partial X,\, X^o} G\ret_{\,X^o} (x,w) = \hat \nu_w(x)\pi(w) \quad
\text{for}\quad x \in \partial X\,,\; w \in X^o.
$$
Therefore
$$
\int_{\partial X} P\ret_{\partial X,\, X^o} G\ret_{\,X^o}^{\,\,2} f\,d\pi =
\sum_{x \in \partial X\,,\, w \in X^o} \hat \nu_w(x)\,\pi(w)\, G\ret_{\,X^o} f(w)
= \int_{X^o} G\ret_{\,X^o} f\,d\pi\,,
$$
since each $\hat \nu_w$ is a probability distribution on $\partial X$. In the same way,
using \eqref{eq:innprod},
$$ 
\int_{\partial X} (I\ret_{\,\partial X} + R)g_1\,d\pi 
= \int_{\partial X} g_1\,d\pi + (\uno\,,R  \bar g_1)_{\pi,\partial X} 
= \int_{\partial X} g_1\,d\pi+ (\uno\,, \bar h_1)_{\pi, X^o} 
= \int_{\partial X} g_1\,d\nu_{\pi}\,,
$$
since $\nu_x = \delta_x$ for $x \in \partial X$. 

\smallskip

Thus, if $f$ and $g_1$ are given and such that \eqref{eq:condition}, 
resp. \eqref{eq:condition'} hold, then Theorem \ref{thm:Poisson} yields
that for arbitrary $z \in \partial X$, the possible choices for $g_2$ are
$$
g_2 = \Bigl(I\ret_{\partial X \setminus \{z\}} - Q\ret_{\,\partial X \setminus \{z\}}\Bigr)^{-1}
\Bigr( P\ret_{\partial X,\, X^o} G\ret_{\,X^o}^{\,\,2} f
+ (I\ret_{\,\partial X} + R)g_1\Bigr) + c\,,
$$
for any $c \in \C$. 

\smallskip

Conversely, if $f$ and $g_2$ are given, and if the matrix 
$I\ret_{\,\partial X} + R$ is invertible -- in particular, in the reversible case --
the function  $g_1$ is determined uniquely as
\begin{equation}\label{eq:DirNeu}
g_1 =  -(I\ret_{\,\partial X} + R)^{-1}\Bigl(\Delta\ret_{\,\,\partial X} g_2 + P\ret_{\partial X,\, X^o} G\ret_{\,X^o}^{\,\,2} f\Bigr).
\end{equation}
In that case, if $f \equiv 0$, then we have the linear Dirichlet to Neumann map 
$L(\partial X) \to L(\partial X)$, $g_2 \mapsto Tg_1\,$ with the \emph{transfer matrix}
\begin{equation}\label{eq:T}
T =  -(I\ret_{\,\partial X} + R)^{-1}\Delta\ret_{\,\,\partial X} 
= (I\ret_{\,\partial X} + R)^{-1}(I\ret_{\,\partial X} -Q).
\end{equation}
Its kernel consists once more of the constant functions, and its
image is the hyperplane 
$$
\Bigl\{ g_1 \in L(\partial X) : \int_{\partial X} g_1 \, d\nu_{\pi} = 0 \Bigr\}.
$$
In any case, the situation for Dirichlet \& Neumann conditions for the bi-Laplacian in
the discrete setting is quite different from the smooth case as considered, e.g.,
by {\sc Gander and Li}~\cite{GL}. See also the Discussion \ref{imp:bi-Green} below.
\end{imp}

We now propose a second approach to the ``plate equation'' via an at first glance
slight modification of the problem \eqref{eq:plate}. 
We decompose $X^o = Y = Y^o\cuplus \partial Y$ according to 
\eqref{eq:sub}. Thus, $Y^o$ is  the ``second interior'' of $X$. Recall from Remarks \ref{rmk:sub} 
the definition of $\Delta_{[Y]}$ and the second way \eqref{eq:normalY*} of defining
the outer normal derivative.  Furthermore, we require that $Y$ is strongly connected,
so that the matrix $P_{[Y]}$ is irreducible, and we write $Q_{[\partial Y]}$ for the transition
matrix of the boundary process on $\partial Y$ according Definition \ref{def:bdryproc} and Lemma
\ref{lem:bdryproc} (with $Y$ in the place of $X$). There is also the associated
matrix $R_{[\partial Y]}$ corresponding to \eqref{eq:RS} and Lemma \ref{lem:invert}.
With these ingredients, we have the following.

\begin{thm}[Discrete plate equation, Variant 2.]\label{thm:plate'}
Let $f \in L(Y^o)$, $g_1 \in L(\partial Y)$ and $g_2 \in L(\partial X)$. 
If the matrix $I_{\partial Y} + R_{[Y]}$ is invertible - in particular, in the reversible
case - the problem to find $u \in L(X)$ such that 
$$
\Delta_{[Y]}^2 u = f \; \text{ on }\; Y^o\,,\quad \partial_{\vec n}^*u = g_1\; \text{ on }
\;\partial Y \AND u = g_2\; \text{ on }\;\partial X
$$
has a unique solution. On $Y$, it is the is the solution of the bi-Laplace Dirichlet problem
$$
\begin{aligned}
\Delta_{[Y]}^2 u &= f \; \text{ on }\; Y^o\,,\quad u = g\; \text{ on } \;\partial Y\,,
\quad \text{where}\\ 
g(y) &= \frac{1}{p(y,\partial X)}\sum_{z \in \partial X} p(y,z)g_2(z) - g_1(y)\,,\quad y \in Y\,,
\end{aligned}
$$
according to Theorem \ref{thm:BiDirichlet}.  
\end{thm}

\begin{proof} 
By \eqref{eq:normalY*} we must have
$$
\partial_{\vec n}^* u(y) = \frac{1}{p(y,\partial X)} \sum_{z \in X \setminus Y}
p(y,z)g_2(z) - u(y)\,, \quad y \in \partial Y.
$$
This yields the values of $u =g$ on $\partial Y$. 
\end{proof}

 We explicitly propose these two variants, which show that the discrete analogue of
typical ``smooth'' equations may be subject to different interpretations; see the discussion
below.

We conclude this Section with another, simple variant of the Dirichlet problem for the 
bi-Laplacian $\Delta^2$ on $X$.

\begin{thm}[Iterated Dirichlet problem.] \label{thm:IterDirichlet}
With $X^o = Y = Y^o\cuplus \partial Y$,  
let $f \in L(Y^o)$, $g_1 \in L(\partial Y)$ and $g_2 \in L(\partial X)$. Then the boundary value problem
$$
\Delta^2 u = f \; \text{ on }\; Y^o\,,\quad \Delta f(y) = g_1 \; \text{ on }\; \partial Y 
\AND u = g_2\; \text{ on }\;\partial X
$$
has a unique solution, which is given by 
$$
\begin{gathered}
u = G\ret_{X^o}G\ret_{\,Y^o}f \,-\, G\ret_{X^o}h_1 \,+\, h_2\,,\quad \text{where}\quad
\\ 
h_1(y) = \int_{\partial Y} g_1\, d\nu_y^{\partial Y} \quad (y \in Y)
\AND h_2(x) = \int_{\partial X} g_2\, d\nu_x^{\partial X} \quad (x \in X).
\end{gathered}
$$
\end{thm}

\begin{proof} Define $v = \Delta u$ on $X^o$. It must solve the Dirichlet problem
$$
\Delta v = f\; \text{ on }\; Y^o \AND v = g_1 \; \text{ on }\; \partial Y\,,
$$
whence by Theorem \ref{thm:Dirichlet},
$v = -G\ret_{\,Y^o}f + h_1\,$. Next, $u$ must solve the  Dirichlet problem
$$
\Delta u = v\; \text{ on }\; X^o \AND v = g_2 \; \text{ on }\; \partial X\,.
$$
The solution is 
$u = -G\ret_{X^o}v + h_2\,$.
\end{proof} 

It is clear that one can iterate further, taking $\Delta^n$ and the $n^{\textrm{th}}$
interior of $X$ on which $f$ is defined (as long as that interior is non-empty), 
as well as the ``onion layers'' of successive boundaries on which the respective boundary 
functions $g_1\,,\dots, g_n$ are specified.

\begin{imp}[Discussion: bi-harmonic Green kernel.]\label{imp:bi-Green}
The \emph{bi-harmonic Green kernel} should be the respective kernel which provides
the solution of the problems considered in the last three theorems,
where all boundary values are set to $0$, and one is given only the function
$f$ defined in the interior. 
In the smooth situation, one of the interesting problems concerns the 
negative part of that kernel corresponding to the plate equation; see \cite{GGS}. 

From our theorems, we see that the choice of that kernel is
case-dependent. 

Let us start with the last one, from Theorem \ref{thm:IterDirichlet}.
The kernel is $\;G\ret_{X^o}G\ret_{\,Y^o}\,$, which is non-negative and
$> 0$ on $X^o \times Y^o$ (as long as $X^o$ is strongly connected).
This corresponds to the bi-harmonic Green kernel of \cite{Ya}. It appears 
not to have a natural counterpart in the classical smooth setting. 

\smallskip

Regarding the two variants concerning the plate equation, we already saw that
the solution of Theorem \ref{thm:plate} it more restrictive. It may 
be natural to consider only the ``first'' boundary. On the other hand, $\Delta$ is not
an infinitesimal operator, and in $\Delta(\Delta u)$, already the first application
of $\Delta$ reaches out to the boundary and involves the boundary values directly.
If we set $g_1 = g_2 = 0$ in \eqref{eq:plate}, then by \eqref{eq:plate-cond}, 
we only get a solution if $f \in L(X^o)$ satisfies
$$
P\ret_{\partial X,\, X^o} G\ret_{\,X^o}^{\,\,2} f = 0.
$$
The associated kernel is then $G\ret_{\,X^o}^{\,\,2}\,$, which is positive,
while it is the function $f$ whose real as well as complex parts (unless they vanish)
must have positive as well as negative values.

\smallskip

In the second variant,  the one of Theorem \ref{thm:plate'}, we have $f \in L(Y^o)$, and 
with the respective boundary values set to zero, the solution is the one of the
bi-Laplace Dirichlet problem $\Delta_{[Y]}^2 u = f$ with $u=0$ on $\partial Y$.
We may equivalently replace this by the bi-Laplace Dirichlet problem on $X$,
$$
\Delta^2 u = f\;\text{ on }\;X^o\,,\;\text{where }\; f \in L(X^o)\,, \AND u=0 \; 
\text{ on }\; \partial X.
$$
We suppose that the matrices  $I_{\partial X} + R$ and $S$ are invertible; see Lemma \ref{lem:invert}.
From Theorem \ref{thm:BiDirichlet}, we get with $K\ret_{\,X^o}$ given by \eqref{eq:SKG}
$$
u = K\ret_{\,X^o}G\ret_{\,X^o}f\,.
$$
Thus, $K\ret_{\,X^o}G\ret_{\,X^o}$ is our bi-harmonic Green kernel for the plate equation in variant 2
(when rewritten in terms of $\Delta_{[Y]}$ and $Y^o$ instead of $\Delta_{[X]}$ 
and $X^o$). In general, it is not positive everywhere on $X^o \times X^o$.
It appears to be a reasonable analogue of the kernel for the smooth plate equation
as in \cite{GGS}.
\end{imp}

\section{Examples}\label{sec:ex}

\noindent
\textbf{A. Simple random walk on an integer interval} \\
In our first example, $X = \{ 0\,,1\,,\dots,N\}$,
the symmetric edges are between successive integers, and we start with symmetric 
edge weights $a(k,k\pm 1)=1$.
The associated Markov chain is reversible with $p(k,k\pm 1) = 1/2$ for $k=1,\dots, N-1$
and $p(0,1)=p(N,N-1)=1$, while all other transition probabilities are $0$.
The stationary probability measure is 
\begin{equation}\label{eq:piSRW}
\pi(k) = \frac{1}{N} \;\text{ for }\; k=1,\dots, N-1, \AND \pi(0)=\pi(N) = \frac{1}{2N}. 
\end{equation}
We set $o= 0$ and $\partial X = \{0,N\}$. Then it is quite easy to compute
the Green kernels $G_{X \setminus \{0\}}$ and $G_{X^o}\,$.
We can use for example the computations of \cite[\S 5.A]{WMarkov}. (Be careful when using Lemma 5.5
of that reference: the $R_k$ there is the $R_{k-1}$ of the subsequent page 119). 

Our $G_{X \setminus \{0\}}$ corresponds to the case when state $0$ is absorbing and state
$N$ is reflecting, and one computes 
$F_{X \setminus \{0\}}(k,m)$, the probability to reach state $m$ when starting at $k\,$:
$$
F_{X \setminus \{0\}}(k,m) = \frac{k}{m} \;\text{ for }\; k\le m, \AND  
F_{X \setminus \{0\}}(k,m) = 1 \;\text{ for }\; k\ge m.
$$
Using the equation of \cite[Thm. 1.38]{WMarkov}, we 
compute for $m=1,\dots, N-1$
$$
\begin{aligned}
G_{X \setminus \{0\}}(m,m) 
&= \frac{1}{1 - \frac12 F_{X \setminus \{0\}}(m-1,m) - \frac12 F_{X \setminus \{0\}}(m+1,m)} 
= \frac{2}{m}\,, \AND\\
G_{X \setminus \{0\}}(N,N) 
&= \frac{1}{1 -  F_{X \setminus \{0\}}(N-1,N)} = \frac{1}{N}\,. 
\end{aligned}
$$
Now using that 
$G_{X \setminus \{0\}}(k,m) = F_{X \setminus \{0\}}(k,m) G_{X \setminus \{0\}}(m,m)$,
we get
\begin{equation}\label{eq:Go}
G_{X \setminus \{0\}}(k,m) = 2\min\{k,m\}   \;\text{ for }\; m < N, \AND
G_{X \setminus \{0\}}(k,N) =k\,. 
\end{equation}
Next, $G_{X^o}$ corresponds to the case when both states $0$ and $N$ are absorbing.
Then, by the same methods, for $k, m \in \{1,\dots, N-1\}$,
$$
F_{X^o}(k,m) = \frac{k}{m} \;\text{ for }\; k\le m, \AND  
F_{X^o}(k,m) = \frac{N-k}{N-m} \;\text{ for }\; k\ge m.
$$
We get 
\begin{equation}\label{eq:G^o}
 G_{X^o}(k,m) = G_{X^o}(m,k) = \frac{2k(N-m)}{N} \;\text{ for }\; 1 \le k\le m \le N-1.
\end{equation}

\noindent\underline{Poisson equation.} Let $f \in L(X)$ with $f(0) + 2\bigl(f(1)+\dots+f(N-1)\bigr)+f(N)=0$.
Then the solution of the Poisson equation grounded at $0$ is 
\begin{equation}\label{eq:uK}
u_{\text{Kirch}}(k) = -G_{X \setminus \{0\}}f(k)=  -2 \sum_{m=1}^{k-1} mf(m) - 2k \sum_{m=k}^{N-1} f(m) - kf(N)\,,
\end{equation}
$k \in \{1,\dots, N\}$.

We skip the Neumann problem, which is equivalent with the Poisson equation.

\smallskip

\noindent\underline{Dirichlet problem.} Let $f \in L(X^o) = L(\{1,\dots, N-1\})$ and 
$g \in L(\partial X) = L(\{0, N\})$. We first find the distributions on the boundary,
using that $\nu_k(N)= F_{X \setminus \{0\}}(k,N)\,$:
\begin{equation}\label{eq:nuk}
\nu_k(0) = \frac{N-k}{N} \AND \nu_k(N) = \frac{k}{N},
\end{equation}
whence
$$
h(k) = \int_{\partial X} g\,d\nu_k = \tfrac{N-k}{N}g(0) + \tfrac{k}{N}g(N)\,.
$$
Then the solution of the Dirichlet problem is
$$
u_{\text{Dir}}(k) = h(k)-G_{X^o}f(k) = h(k) - \frac{2(N-k)}{N}\sum_{m=1}^{k-1}mf(m)   
- \frac{2k}{N}\sum_{m=k}^{N-1}(N-m)f(m).
$$

\noindent\underline{Mixed problem.} The simplest mixed problem is when in addition to 
$f \in L(X^o)$ we require that the solution satisfies 
$$
u(0) = g(0) \AND \partial_{\vec n}u(N) = g(N)\,, \quad\text{where}\quad  g(0), g(N) \in \C\,. 
$$
In this case, we extend $f$ to $\{ 1,\dots,N\}$ by setting $f(N) = - g(N)$.
Then the solution is $u(k) = g(0) + u_{\text{Kirch}}(k)$, where the latter is
given by \eqref{eq:uK}.
 
\smallskip

\noindent\underline{Poisson and Dirichlet problem with potential.} We only consider the 
easiest case,  when the potential $v$ is constant. We set $\lambda = 1+ v$,
so that our assumption is $|\lambda|> 1$. Then $\wt P = \frac{1}{\lambda} P$ on $X$.
Computing $\wt G = \sum_{n=0}^{\infty} \frac{1}{\lambda^n} P^n\,$ amounts to invert
a tri-diagonal matrix. There are various ways. We used once more  
\cite[5.A]{WMarkov}. 
Let $Q_k(\lambda)$ and $R_k(\lambda)$ be the $k^{\text{th}}$ Chebyshev polynomials 
of the first and second kind, respectively, that is,
$$
Q_k(\cos \varphi) = \cos k\varphi \AND R_k(\cos \varphi) = \frac{\sin (k+1)\varphi}{\sin \varphi}.
$$
It will be convenient to set $R_{-1}(\lambda)=0$. 
After some manipulations, 
setting $\epsilon_m = 2$ for $m \in \{1,\dots,N-1\}$ and $\epsilon_0=\epsilon_N = 1$, we get 
for $k,m \in X$
$$
\wt G(k,m) = \begin{cases} 
 \dps
\epsilon_m \frac{\lambda}{\lambda^2-1}\frac{Q_k(\lambda)Q_{N-m}(\lambda)}{R_{N-1}(\lambda)}\,,&\text{if }\; k \le m\,,\\[11pt]
 \dps
\epsilon_m \frac{\lambda}{\lambda^2-1}\frac{Q_m(\lambda)Q_{N-k}(\lambda)}{R_{N-1}(\lambda)}\,,&\text{if }\; k \ge m\,.
\end{cases}
$$
Then, for $f \in L(\{0,\dots, N\})$, the unique solution $u$ of the Poisson
problem with constant potential $\Delta u - (\lambda-1)u = f$ is
$$
u_{\text{Kirch}}(k) = -\sum_{k=0}^N \wt G(k,m)\frac{f(m)}{\lambda}.
$$
Next, with $X^o = \{ 1, \dots, N-1\}$, we can also compute $\wt G_{X^o}(k,m)$ for $k,m \in X^o$
as follows.
$$
\wt G_{X^o}(k,m) = \begin{cases} 
 \dps
2\lambda\frac{R_{k-1}(\lambda)R_{N-m-1}(\lambda)}{R_{N-1}(\lambda)}\,,&\text{if }\; k \le m\,,\\[11pt]
 \dps
2\lambda\frac{R_{m-1}(\lambda)R_{N-k-1}(\lambda)}{R_{N-1}(\lambda)}\,,&\text{if }\; k \ge m\,.
\end{cases}
$$
Next we compute the measures $\tilde \nu_k$. (In  \cite[5.A]{WMarkov}, $\tilde \nu_k(0)$
corresponds to the quantity $F(k,0|z)$ in the middle of p. 119, with $z=1/\lambda$.)
$$
\tilde \nu_k(0) = \frac{R_{N-k-1}(\lambda)}{R_{N-1}(\lambda)} \AND 
\tilde \nu_k(N) = \frac{R_{k-1}(\lambda)}{R_{N-1}(\lambda)}.
$$
These are in general not probability measures on $\{ 0,N\}$, and indeed not even
necessarily positive, unless $\lambda > 1$ is real.
Given $f \in L(\{1,\dots,N-1\})$ and $g \in L(\{0,N\})$,
the unique solution $u$ of the Dirichlet
problem with constant potential $\Delta u - (\lambda-1)u = f$ on $X^o$, $u=g$ on $\partial X$
is now
$$
u_{\text{Dir}}(k) = g(0) \tilde \nu_k(0) + g(N) \tilde \nu_k(N) -
\sum_{m=1}^{N-1} \wt G_{X^o}(k,m)\frac{f(m)}{\lambda}.
$$

\smallskip 

\noindent\underline{Robin problem.} The simplest case is $\alpha = \beta \ne 0$, constant on 
$\partial X = \{ 0, N\}$. Referring to Corollary \ref{cor:Robin}, $B=\emptyset$ and 
$X^o_{\beta} = X$ with $\tilde p(k,k\pm 1) = 1/2$ whenever $k, k \pm 1 \in \{0,\dots, N\}$.

Computing the associated Green kernel $\wt G_{\text{Rob}} = (I - \wt P)^{-1}$ is analogous to 
computing $G_{X^o}$ in \eqref{eq:G^o}. Indeed, instead of adding one ``tomb'' state,
consider the extended space $\wt X = \{-1, 0, \dots, N, N+1 \}$ with the simple random walk as
before and the new boundary $\{-1, N+1\}$ and interior $\wt X^o = X$.  Thus, when we shift
the elements by $1$ and replace $N$ with $N+2$, we are back to the computation of  
$G_{\wt X^o}$. Hence, for $k, m \in X = \{0,\dots, N\}$,
$$
G_{\text{Rob}}(k,m) = \frac{2(k+1)(N+1-m)}{N+2} \;\text{ for }\; 
0 \le k\le m \le N.
$$
Therefore, given $f \in L(\{1,\dots, N-1\})$ and $g \in  L(\{0, N\})$, 
the unique $u= u_{\text{Rob}} \in L(\{0,\dots, N\})$ such that $\Delta u = f$ on 
$\{1,\dots, N-1\}$ and $\alpha\cdot (u +  \partial_{\vec n}u) = g$ on $\{0, N\}$ is
$$
u_{\text{Rob}}(k) = G_{\text{Rob}}(k,0)\frac{g(0)}{2\alpha} 
+ G_{\text{Rob}}(k,N)\frac{g(N)}{2\alpha} - \sum_{m=1}^{N-1} G_{\text{Rob}}(k,m)
f(m)\,.
$$

\smallskip 

\noindent
\underline{The bi-Laplacian.} We have computed
the Green kernels $G_{X \setminus \{0\}}$ and $G_{X^o}$ in \eqref{eq:Go} and \eqref{eq:G^o},
as well as the hitting distributions $\nu_k$ in \eqref{eq:nuk}.
On this basis, the solutions of the iterated Poisson equation and the bi-Laplace 
Neumann problem can be written down immediately.
So we next compute the transition matrix $Q$ of the boundary chain of Definition 
\ref{def:bdryproc} and Lemma \ref{lem:bdryproc}, and the matrix $R$
of \eqref{eq:RS}.

We have $q(0,0) = p(0,1)\nu_1(0) = \nu_1(0) = \frac{N-1}{N} = q(N,N)$,
and we get 
$$
Q = \frac{1}{N}\begin{pmatrix} N-1 & 1 \\ 1 & N-1    \end{pmatrix}.
$$
With $\pi$ given by \eqref{eq:piSRW}, we have for $i, j \in \{0,N\}$
$$
r(i,j) = 2\sum_{k=1}^{N-1} \nu_k(i)\nu_k(j). 
$$
We get 
$$
R = \frac{N-1}{3N}\begin{pmatrix} 2N-1 & N+1 \\ N+1 & 2N-1 \end{pmatrix} 
\AND  
(I+R)^{-1} = \frac{1}{N^3 + 2N}\begin{pmatrix} 2N^2+1 & 1-N^2 \\ 1-N^2 & 2N^2+1 \end{pmatrix} . 
$$
We now want to consider the bi-Laplace Dirichlet problem with $f \equiv 0$ on $X^o$ and
boundary function $g=g_2 \in L(\{0,N\})$. Since the boundary has only 2 elements, we prefer to use the transfer
matrix for the bi-Laplace Dirichlet to Neumann map according to \eqref{eq:T}, which is
$$
T = \frac{3}{N^2 + 2}\begin{pmatrix} 1 & -1 \\ -1 & 1 \end{pmatrix}.
$$
In other words, if the Dirichlet boundary values are $g(0)$ and $g(N)$,
then the solution must have the Neumann boundary values 
$g_1(0) = -g_1(N)=  3\bigl((g(0)-g(N)\bigr)\big/\bigl(N^2 + 2\bigr)$.
We get 
$$
\begin{aligned}
h_1(k) &= \int_{\partial X} g_1\,d\nu_k = \frac{3N-6k}{N^3 + 2N}\bigl(g(0)-g(N)\bigr)
\AND \\  h_2(k) &= \int_{\partial X} g\,d\nu_k = \frac{N-k}{N}g(0) +  \frac{k}{N}g(N).
\end{aligned}
$$
The unique solution $u$ of the bi-Laplace Dirichlet problem
$$
\Delta u = 0 \; \text{ on }\; \{1,\dots, N\}\,,\quad u(0) = g(0)\,,\; u(N) = g(N)
$$
is now given via \eqref{eq:G^o} as
$$
u(k) = h_2(k) +  \frac{2(N-k)}{N}\sum_{m=1}^{k-1}m\,h_1(m)   
- \frac{2k}{N}\sum_{m=k}^{N-1}(N-m)h_1(m).
$$
The other bi-Laplace equations are obtained along the same lines:
The plate equation of Theorem \ref{thm:plate'} (with $f \equiv 0$)
is a variant of what we have just
computed, replacing $\{0,\dots, N\}$ with $Y = \{1, \dots, N-1\}$
and $\partial Y = \{ 1, N-1\}$.
The iterated Dirichlet problem of Theorem \ref{thm:IterDirichlet} 
means that one has to apply the Green kernel $G_{X^o}$ computed in
\eqref{eq:G^o} as well as the Green kernel $G_{Y^o}$, which is computed
in the same way (shifting down by $1$ and replacing $N$ by $N-2$). 

\bigskip

\noindent
\textbf{B. A non-reversible example.}\\ We set $X = \{ 1, \dots, N\}$, where $N \ge 3$, 
and choose probabilities $p_1\,,\dots, p_N > 0$ with sum~$1$.
The transition probabilities are then
$$
p(1,k) = p_k\;\text{ for }\; k \in X, \AND p(k,k-1)=1\;  \text{ for }\; k\ge 2.
$$
As the boundary, we choose $\partial X = \{N-1, N\}$, and as the root, we choose 
$o = 1$. The stationary probability distribution is 
\begin{equation}\label{eq:pi-nonr}
\pi(k) = \sum_{m=k}^N p_m \Big/ \sum_{m=1}^N mp_m\,.
\end{equation}

\noindent\underline{Poisson equation.} It is easy to compute $G_{X \setminus \{1\}}$:
\begin{equation}\label{eq:GX-1}
G_{X \setminus \{1\}}(k,m) = 1 \;  \text{ for }\; 2 \le m \le k\,,
\AND G_{X \setminus \{1\}}(k,m) = 0 \;  \text{ otherwise}.
\end{equation}
Therefore, given $f \in L(X)$ with $\int_X f\,d\pi = 0$, the unique solution of
the Poisson equation grounded at $1$ is
$$
u(k) = \sum_{m=2}^k f(m).
$$
For the remaining issues,
we compute the Green kernel $G(k,m|z) = \sum_{n=0}^{\infty} p^{(n)}(k,m)z^{n}$, following the
methods of \cite{WMarkov}. The computations are also valid when $p_1 + \dots + p_N < 1$.
First, consider $F(k,m|z) = G(k,m|z)/G(m,m|z)$, the generating function of the first
hitting probability at $m$, when the Markov chain starts at $k$.
We have 
$$
F(k,m|z) = z^{k-m}\; \text{ for }\; k \ge m \AND  F(k,m|z) = z^{k-1}F(1,m|z) 
\; \text{ for }\; k < m.
$$
Next, 
$$
F(1,m|z) = \sum_{j=1}^N p_jz F(k,m|z) 
= \sum_{j=1}^{m-1} p_jz^j F(1,m|z) + \sum_{j=m}^N p_j z^{1+j-m}\; \text{ for }\; m>1.
$$
We conclude that
$$
F(k,m|z) = \sum_{j=m}^N p_jz^{j+k-m} \bigg/ \biggl( 1- \sum_{j=1}^{m-1} p_jz^j \biggr) 
$$
Next, the general formula $G(m,m|z) = 1 \Big/ \Bigl( 1 - \sum_j p(m,j)z \,F(j,m|z) \Bigr)$
yields
$$
G(m,m|z)= \biggl( 1- \sum_{j=1}^{m-1} p_jz^j \biggr) \bigg/ 
\biggl( 1- \sum_{j=1}^{N} p_jz^j \biggr).
$$
Altogether,
\begin{equation}\label{eq:GX^o'}
G(k,m|z) = \begin{cases}
  \dps \biggl( z^{k-m}- \sum_{j=1}^{m-1} p_jz^{j+k-m} \biggr) \bigg/ 
  \biggl( 1- \sum_{j=1}^{N} p_jz^j \biggr) \;&\text{for }\; k \ge m, \AND\\[5pt]
  \hspace*{1.3cm}\dps \biggl( \,\sum_{j=m}^{N} p_jz^{j+k-m} \biggr) \bigg/ 
  \biggl(1- \sum_{j=1}^{N} p_jz^j \biggr) \;&\text{for }\; k < m.
  \end{cases}
\end{equation}

\noindent\underline{Dirichlet problem.} Let $f \in L(X^o) = L(\{1,\dots, N-2\})$ and 
$g \in L(\partial X) = L(\{N-1, N\})$. The Green kernel $G_{X^o}$ is obtained from
\eqref{eq:GX^o'} by replacing $N$ with $N-2$ and setting $z=1$. 
Then
\begin{equation}\label{eq:G^o-nonr}
G_{X^o}(k,m) = 
\begin{cases} \dfrac{\pi(m)}{\pi(N-1)} \;  &\text{for }\; 1 \le k \le m \le N-2\,,\AND\\[5pt] 
\dfrac{\bigl(\pi(m) - \pi(N-1)\bigr)}{\pi(N-1)} \;  &\text{for }\; 1 \le m < k \le N-2\,.
\end{cases}
\end{equation}
We next compute the hitting distributions $\nu_k$ on the boundary. It is clear that
For all $k \le N-2$ and $j \in \{N-1,N\}$,
$$
\nu_k(j) = \nu_1(j) = G_{X^o}(1,1)\,p_j = \frac{\pi(1)}{\pi(N-1)}p_j = \frac{p_j}{p_{N-1} + p_N} \,.
$$
Thus, given $g \in L(\{N-1,N\})$, the associated harmonic function on $X^o$ 
is
$$
h(k) \int_{\partial X}  g\,d\nu_k = h(1) = \frac{p_{N-1}g(N-1) + p_Ng(N)}{p_{N-1} + p_N}\,, 
\quad k \le N-2,
$$
while of course $h(j) = g(j)$ for $j \in \{ N-1,N\}$.
If now in addition $f \in L(\{1,\dots, N-2\})$, then the solution of the Dirichlet problem
is
$$
u_{\text{Dir}}(k) = \frac{p_{N-1}g(N-1) + p_Ng(N)}{p_{N-1} + p_N}
- \frac{1}{\pi(N-1)}\sum_{m=1}^{N-2}\pi(m) f(m) + \sum_{m=1}^{k-1} f(m)\,, \quad k \le N-2.
$$

\noindent\underline{Mixed problem.} A natural mixed problem is when in addition to 
$f \in L(X^o)$ we require that the solution satisfies 
$$
u(N-1) = g(N-1) \AND \partial_{\vec n}u(N) = g(N)\,, \quad\text{where}\quad  
g(N-1), g(N) \in \C\,. 
$$
Then we must have $u(N) = g(N) + g(N-1)$, and this time we do not re-conduct the problem
to the Poisson equation. Instead, we are lead to the above Dirichlet problem with boundary 
function $\tilde g(N-1) = g(N-1)$ and $\tilde g(N) = g(N-1)+g(N)$.

\smallskip

\noindent\underline{Poisson equation with potential.} Again, we only consider the
easiest case,  when the potential $v$ is constant, with $\lambda = 1+ v$,
so that  $|\lambda|> 1$. Again, $\wt P = \frac{1}{\lambda} P$ on $X$, and the 
associated Green kernel $\wt G$ is the one of \eqref{eq:GX^o'}, with $z = 1/\lambda$.
This leads to the solution.

\smallskip

We skip the Dirichlet problem with potential and the Robin problem.

\smallskip

\noindent
\underline{The bi-Laplacian.} Again,  the solutions of the iterated Poisson 
problem and the bi-Laplace Neumann problem can be written down immediately
via the Green kernel $G_{X \setminus \{ 1 \}}$ computed in \eqref{eq:GX-1}.
We also have $G_{X^o}$ and the hitting distributions on the boundary.

It is easy to compute the transition matrix of the boundary process:
$q(N,N-1)=1$ and $q(N-1,j) = \nu_1(j)$, $j \in \{ N-1,N\}$.
In this example, the exit boundary is $\{ N-1, N\}$, but the entrance boundary is 
only $\{ N-1\}$, that is,
$\hat \nu_k(N-1)=1$ and $\hat \nu_k(N)=0$ for all $k \in \{1, \dots, N-2\}$.

With $\pi$ given by \eqref{eq:pi-nonr}, we have for $i, j \in \{N-1,N\}$
$$
r(i,j) = \frac{1}{\pi(i)} \sum_{k=1}^{N-2} \hat\nu_k(i)\nu_k(j)\pi(k)
= \frac{1}{\pi(i)}\sum_{k=1}^{N-2} \delta_{N-1}(i)\nu_1(j)\pi(k).
$$
Thus, 
$$
\begin{gathered}
Q = \begin{pmatrix} \frac{\pi(1)}{\pi(N-1)}p_{N-1} & \frac{\pi(1)}{\pi(N-1)}p_N \\ 1 & 0 
                        \end{pmatrix}
\AND
R = \begin{pmatrix} C\,p_{N-1} & C\,p_N \\ 0 & 0 
                        \end{pmatrix}\,,\quad \text{where}\\ 
    C = \frac{\pi(1)\bigl(1-\pi(N-1)-\pi(N)\bigr)}{\pi(N-1)^2}.
\end{gathered}    
$$
Now the transfer
matrix for the Bi-Laplace Dirichlet to Neumann map is
$$
T = \begin{pmatrix} D & -D \\ -1 & 1 \end{pmatrix}, \quad \text{where }\;
D= \frac{\pi(N)^2 - \pi(N) + 2\pi(N)\pi(N-1)}{\pi(N)^2 - \pi(N) + \pi(N-1)}.
$$
We remark that $D < 0$.
If the Dirichlet boundary values are $g(N-1)$ and $g(N)$,
then the solution must have the Neumann boundary values 
$g_1(N-1) = D \bigl(g(N-1)-g(N)\bigr)$ and $g_1(N) = g(N)-g(N-1)$.
We get for $k \in \{ 1, \dots, N-2\}$
$$
\begin{aligned}
h_1(k) &= \int_{\partial X} g_1\,d\nu_k = h_1(1) = 
\bigl(g(N)-g(N-1)\bigr)\frac{p_N-D p_{N-1}}{p_{N-1}+p_N}
\AND \\  h_2(k) &= \int_{\partial X} g\,d\nu_k = h_2(1) 
= \frac{g(N-1)p_{N-1} + g(N)p_N}{p_{N-1}+p_N}.
\end{aligned}
$$
The unique solution $u$ of the bi-Laplace Dirichlet problem
$$
\Delta u = 0 \; \text{ on }\; \{1,\dots, N-2\}\,,\quad u(N-1) = g(N-1)\,,\; u(N) = g(N)
$$
is now given via \eqref{eq:G^o-nonr} as
$$
u(k) = \left(\frac{1-\pi(N-1)-\pi(N)}{\pi(N-1)} - (k-1)\right) h_1(1) + h_2(1).
$$
With this computation, we end the second example.


\end{document}